\documentclass[a4paper,11pt]{amsart}

\makeatletter
\@namedef{subjclassname@2010}{%
  \textup{2010} Mathematics Subject Classification}
\makeatother

\usepackage{amscd}
\usepackage{amsmath}
\usepackage{amssymb}
\usepackage{amsthm}
\usepackage{float}
\usepackage[dvips]{graphicx}
\usepackage{enumerate}

\usepackage{array}
\usepackage{amscd}
\usepackage[all]{xy}
\usepackage{url}

\DeclareFontFamily{U}{rsfs}{%
\skewchar\font127}
\DeclareFontShape{U}{rsfs}{m}{n}{%
<-6>rsfs5<6-8.5>rsfs7<8.5->rsfs10}{}
\DeclareSymbolFont{rsfs}{U}{rsfs}{m}{n}
\DeclareSymbolFontAlphabet
{\mathrsfs}{rsfs}
\DeclareRobustCommand*\rsfs{%
\@fontswitch\relax\mathrsfs}

\theoremstyle{plain}
\newtheorem{thm}{Theorem}[section]
\newtheorem{prop}[thm]{Proposition}
\newtheorem{lem}[thm]{Lemma}
\newtheorem{sublem}[thm]{Sublemma}
\newtheorem{defi}[thm]{Definition}
\newtheorem{rmk}[thm]{Remark}
\newtheorem{aside}[thm]{Aside}

\newtheorem{cor}[thm]{Corollary}

\newtheorem{prop-defi}[thm]{Proposition-Definition}
\newtheorem{thm-defi}[thm]{Theorem-Definition}
\newtheorem{lem-defi}[thm]{Lemma-Definition}

\newtheorem{assum}[thm]{Assumption}
\newtheorem{conj}[thm]{Conjecture}
\newtheorem{exam}[thm]{Example}

\newdimen\argwidth
\def\db[#1\db]{
 \setbox0=\hbox{$#1$}\argwidth=\wd0
 \setbox0=\hbox{$\left[\box0\right]$}
  \advance\argwidth by -\wd0
 \left[\
 n.3\argwidth\box0 \kern.3\argwidth\right]}

\newcommand{\vt}{\vartheta}
\newcommand{\aA}{\mathcal{A}}
\newcommand{\bB}{\mathcal{B}}
\newcommand{\cC}{\mathcal{C}}
\newcommand{\dD}{\mathcal{D}}
\newcommand{\eE}{\mathcal{E}}
\newcommand{\fF}{\mathcal{F}}

\newcommand{\hH}{\mathcal{H}}

\newcommand{\lL}{\mathcal{L}}
\newcommand{\mM}{\mathcal{M}}

\newcommand{\oO}{\mathcal{O}}
\newcommand{\pP}{\mathcal{P}}
\newcommand{\qQ}{\mathcal{Q}}

\newcommand{\sS}{\mathcal{S}}
\newcommand{\tT}{\mathcal{T}}
\newcommand{\uU}{\mathcal{U}}

\newcommand{\wW}{\mathcal{W}}
\newcommand{\xX}{\mathcal{X}}
\newcommand{\yY}{\mathcal{Y}}

\newcommand{\Hom}{\mathop{\rm Hom}\nolimits}

\newcommand{\dR}{\mathbf{R}}
\newcommand{\dL}{\mathbf{L}}
\newcommand{\NS}{\mathop{\rm NS}\nolimits}

\newcommand{\id}{\textrm{id}}

\newcommand{\ch}{\mathop{\rm ch}\nolimits}

\newcommand{\Ext}{\mathop{\rm Ext}\nolimits}
\newcommand{\Spec}{\mathop{\rm Spec}\nolimits}

\newcommand{\Coh}{\mathop{\rm Coh}\nolimits}

\newcommand{\im}{\mathop{\rm im}\nolimits}

\newcommand{\cneq}{\mathrel{\raise.095ex\hbox{:}\mkern-4.2mu=}}
\newcommand{\eqcn}{\mathrel{=\mkern-4.5mu\raise.095ex\hbox{:}}}

\newcommand{\Aut}{\mathop{\rm Aut}\nolimits}

\newcommand{\Stab}{\mathop{\rm Stab}\nolimits}

\newcommand{\DT}{\mathop{\rm DT}\nolimits}

\newcommand{\modu}{\mathop{\rm mod}\nolimits}

\newcommand{\Slice}{\mathop{\rm Slice}\nolimits}

\newcommand{\Imm}{\mathop{\rm Im}\nolimits}

\newcommand{\Ker}{\mathop{\rm ker}\nolimits}
\newcommand{\Ree}{\mathop{\rm Re}\nolimits}

\newcommand{\cl}{\mathop{\rm cl}\nolimits}

\newcommand{\HN}{\operatorname{HN}}

\begin{document}

\title[Moduli of semistable objects and DT invariants]{Moduli of Bridgeland semistable objects  on 3-folds and Donaldson-Thomas invariants}

\date{\today}

\author{Dulip Piyaratne and Yukinobu Toda}

\address{Kavli Institute for the Physics and Mathematics of the Universe, University of Tokyo,
5-1-5 Kashiwanoha, Kashiwa, 277-8583, Japan.}

\email{dulip.piyaratne@ipmu.jp}

\address{Kavli Institute for the Physics and Mathematics of the Universe, University of Tokyo,
5-1-5 Kashiwanoha, Kashiwa, 277-8583, Japan.}

\email{yukinobu.toda@ipmu.jp}



\begin{abstract}
We show that
the moduli stacks of Bridgeland semistable 
objects on smooth projective 3-folds are 
quasi-proper algebraic stacks 
of finite type, if they satisfy 
the Bogomolov-Gieseker (BG for short) inequality 
conjecture proposed by Bayer, Macr\`{i}
and the second author. 
The key ingredients are the equivalent form 
of the BG inequality conjecture and its generalization 
to arbitrary very weak stability conditions. 
This result is applied to define
Donaldson-Thomas  invariants counting Bridgeland 
semistable objects on smooth projective 
Calabi-Yau 3-folds satisfying the BG inequality conjecture, 
for example on \'etale quotients 
of abelian 3-folds. 
\end{abstract}

\maketitle

\section{Introduction}
\subsection{Motivation and background}
Let $X$ be a smooth projective variety over $\mathbb{C}$. 
Following Douglas's 
work~\cite{Dou2} on $\Pi$-stability in physics, 
Bridgeland introduced  the 
complex manifold
$$
\Stab(X)
$$
called the \textit{space of stability conditions} on the bounded derived category $D^b \Coh(X)$ of 
coherent sheaves on $X$ (see \cite{Brs1}). 
Each point of $\Stab(X)$
determines certain semistable 
objects in $D^b \Coh(X)$, and 
the whole space is expected to contain 
the stringy K\"{a}hler moduli space of $X$. 
Also the space $\Stab(X)$ has potential 
applications to Donaldson-Thomas (DT for short) invariants~\cite{Tcurve1}, \cite{Tcurve2}, \cite{TodBG}, 
\cite{TICM}, 
and birational geometry~\cite{BaMa2}, \cite{BaMa3}, 
\cite{ABCH}, \cite{Todext}, \cite{Todbir}.
Here one needs to construct 
the moduli spaces of Bridgeland semistable 
objects, and then to study the variations of these moduli 
spaces under the changes of stability conditions. 

However, in general the space of stability conditions is
a difficult object to
study, and there exist several foundational issues. 
At least, the following conjecture needs to be be settled at firsthand:
 
\begin{conj}\label{conj:intro}
We have the following for $X$:
\begin{enumerate}
\item $\Stab(X) \neq \emptyset$, and
\item for each stability condition $\sigma \in \Stab(X)$, the moduli stack of 
$\sigma$-semistable objects with a fixed Chern character
is a quasi-proper
algebraic stack of finite type.  
\end{enumerate}
\end{conj}
Here 
an algebraic stack is called \textit{quasi-proper}
if it satisfies the valuative criterion of 
properness
in~\cite[Proposition~2.44]{GT}
without the separatedness. 
The above conjecture is known to be true 
when $\dim X \le 2$. 
The construction problem (i) 
for surfaces was solved by 
Bridgeland~\cite{Brs2}
and Arcara-Bertram~\cite{AB}, by
tilting coherent sheaves. 
However, this is an open problem in $\dim X \ge 3$. 
For the 3-dimensional case, Bayer, Macr\`{i} and the second 
author~\cite{BMT} reduced the problem (i)
to a conjectural Bogomolov-Gieseker (BG for short) type 
inequality involving the Chern characters of 
certain two term complexes, called \textit{tilt semistable objects} (see Conjecture \ref{intro:BMT} below). 
On the other hand, the moduli problem (ii)
was solved by the second author~\cite{Tst3}
for K3 surfaces, and the same argument 
also applies to any surface.
The main 
purpose of this paper is to 
solve problem (ii) 
for 3-folds
satisfying the above mentioned BG inequality conjecture
in~\cite{BMT}. 
A rough statement is as follows:
\begin{thm}\label{thm:intro0}
Let $X$ be a smooth projective 3-fold 
satisfying the BG inequality conjecture in~\cite{BMT}. 
Then it satisfies Conjecture~\ref{conj:intro} (ii). 
\end{thm}
The argument proving Theorem~\ref{thm:intro0}
also proves the same statement for tilt semistable 
objects on 3-folds without assuming the BG inequality conjecture, 
which was also studied in~\cite[Proposition~3.7]{TodBG}. 
So far the BG inequality conjecture in~\cite{BMT} 
is known to hold in the following cases:
\begin{itemize}
\item $X=\mathbb{P}^3$ by Macr\`{i}~\cite{MaBo}. 
\item $X \subset \mathbb{P}^4$
is a smooth quadric threefold by Schmidt~\cite{BSch}.
\item 
$X$ is an abelian 3-fold by Maciocia and the first
author~\cite{MaPi1}, \cite{MaPi2}, \cite{Piya}, and  by Bayer, Macr\`{i} and Stellari~\cite{BMS}.
\item $X$ is an \'etale quotient of an abelian 3-fold by Bayer, Macr\`{i} and Stellari~\cite{BMS}. 
\end{itemize}
The result of Theorem~\ref{thm:intro0} is applied 
to the above 3-folds, and it gives new non-trivial 
Bridgeland moduli spaces on 3-folds. 
Furthermore, we use Theorem~\ref{thm:intro0} to construct 
Donaldson-Thomas invariants 
counting Bridgeland semistable objects on Calabi-Yau 3-folds
satisfying the BG inequality conjecture, 
for example on A-type Calabi-Yau 3-folds~\cite{BMS}, 
fulfilling the expected properties. 

\subsection{BG type inequality conjecture}
Let $X$ be a smooth projective 3-fold.
Let $B \in \mathrm{NS}(X)_{\mathbb{Q}}$ and $\omega \in \mathrm{NS}(X)_{\mathbb{R}}$ 
be an ample class with $\omega^2$ rational; that is $\omega = mH$ for some ample divisor class $H \in \mathrm{NS}(X)$ with $m^2 \in \mathbb{Q}_{>0}$.
In \cite{BMT}, Bayer, Macr\`{i} and the second author
constructed data 
\begin{align}\label{intro:sigma}
\sigma_{\omega, B}=(Z_{\omega, B}, \aA_{\omega, B})
\end{align}
for a conjectural Bridgeland stability condition on $X$. 
Here $\aA_{\omega, B}$ is the
heart of a bounded t-structure on $D^b \Coh(X)$ 
given as a double tilt of $\Coh(X)$,  
and $Z_{\omega, B} \colon K(X) \to \mathbb{C}$
is the group homomorphism defined by
\begin{align*}
Z_{\omega, B}(E) \cneq -\int_{X} e^{-i\omega} \ch^B(E),
\end{align*}
where $\ch^B(E) \cneq e^{-B}\ch(E)$.
The pair (\ref{intro:sigma}) 
is shown to 
give a Bridgeland stability condition
on $D^b \Coh(X)$, if the
following conjectural BG inequality 
holds for certain two term complexes, called tilt 
semistable objects
(see Conjecture~\ref{conj:BMT}):
\begin{conj}\emph{(BG Inequality Conjecture,  \cite{BMT})}\label{intro:BMT}
For a tilt semistable object
$E \in D^b \Coh(X)$ 
with $\Imm Z_{\omega, B}(E)=0$,  
we have the inequality
\begin{align*}
\ch_3^B(E) \le \frac{1}{18} \omega^2 \ch_1^B(E). 
\end{align*}
\end{conj} 

One of the important statements in~\cite{BMS}
is that, when $B$ and $\omega$ are 
proportional, 
Conjecture~\ref{intro:BMT} is equivalent to 
another conjectural inequality for tilt semistable 
objects $E$ without the condition 
$\Imm Z_{\omega, B}(E)=0$. 
Our proof of Conjecture~\ref{conj:intro} (ii) 
for 3-folds rely on this equivalent inequality, which 
we generalize to the case when $B$ and $\omega$ are not 
proportional: 
\begin{thm}\emph{(\cite{BMS}, $B$ and $\omega$ are proportional; 
Theorem~\ref{thm:BG:equiv}, in general)}\label{intro:equiv}
Conjecture~\ref{conj:BMT} is equivalent to the following: 

Any tilt semistable object $E$ satisfies the inequality
\begin{align}\notag
(\omega^2\ch_1^B(E))^2
&-2\omega^3\ch_0^B(E) \omega\ch_2^B(E)  \\ \label{intro:BG2}
&+12 (\omega\ch_2^B(E))^2 -18 \omega^2\ch_1^B(E) \ch_3^B(E) \ge 0. 
\end{align}
\end{thm}
The advantage of the inequality (\ref{intro:BG2})
is that it can be used to imply the support property of $\sigma_{\omega, B}$
in some sense. More precisely, Conjecture~\ref{intro:BMT} implies  
\begin{align*}
\sigma_{\omega, B} \in 
\Stab_{\omega, B}^{\circ}(X),
\end{align*}
where $\Stab_{\omega, B}^{\circ}(X)$ is
a connected component of the 
space of stability conditions 
whose central charges are written as 
linear combinations of 
$\omega^{3-j}\ch_j^B(E)$ for $0\le j\le 3$. 
The inequality (\ref{intro:BG2})
also plays a key role 
to generalize our moduli problem for an arbitrary 
very weak stability condition and associated tilting, 
which we discuss in this paper.

\subsection{Moduli stacks of Bridgeland semistable objects on 3-folds}

Now we give the precise statement of Theorem~\ref{thm:intro0} 
as follows: 
\begin{thm}\emph{(Theorem~\ref{thm:stack})}\label{thm:intro:stack}
Suppose that $X$ is a smooth projective 
3-fold satisfying Conjecture~\ref{intro:BMT}.
Then for any $\sigma=(Z, \aA) \in \Stab_{\omega, B}^{\circ}(X)$, 
the moduli stack of $\sigma$-semistable objects $E\in \aA$
with a fixed 
vector 
\begin{align*}
(\omega^3\ch_0^B(E), \omega^2 \ch_1^B(E), \omega \ch_2^B(E), \ch_3^B(E))
\end{align*}
is a quasi-proper algebraic stack 
of finite type.   
\end{thm} 
The strategy of 
the proof of Theorem~\ref{thm:intro:stack} is as follows. 
In~\cite{Tst3}, the second author reduced the moduli problem 
of Bridgeland semistable objects to the 
following two problems: 
\begin{enumerate}
\item 
\textit{generic flatness} of the corresponding heart. 
\item \textit{boundedness of semistable objects}. 
\end{enumerate}
We prove the properties (i), (ii)
 for the stability condition $\sigma_{\omega, B}$
defined in terms of data (\ref{intro:sigma}),
and show that they are preserved under the deformations of 
stability conditions. 
In order to show (i), (ii) for $\sigma_{\omega, B}$, 
we generalize the 
tilting construction 
which appeared in the construction of Bridgeland 
stability on surfaces, tilt stability and 
its further tilting.  
Following~\cite{BMS}, 
we introduce the notion of very weak stability
conditions (see Definition~\ref{defi:veryweak})
 on $D^b \Coh(X)$, which 
generalizes classical slope stability on sheaves
 and tilt stability on two term complexes of sheaves. 
Roughly speaking, a very weak stability condition
consists of a pair $(Z, \aA)$ as similar to Bridgeland stability,
 but we allow some objects in $\aA$
are mapped to zero by $Z$.  
For a very weak stability condition $(Z, \aA)$, we introduce
the notion of its BG type inequality in the form
\begin{align}\label{intro:BG}
\Ree Z(E) \Delta_R(E) + \Imm Z(E) \Delta_I(E) \ge 0
\end{align}
for any semistable object $E \in \aA$
such that $\Delta_R, \Delta_I$ satisfying certain conditions. 
The inequality (\ref{intro:BG}) turns out 
to be a generalization of the classical 
BG inequality for surfaces, and the 
conjectural inequality (\ref{intro:BG2}) for tilt semistable objects. 
We construct the tilting of $(Z, \aA)$ as a one parameter 
family of very weak stability conditions of the form
\begin{align}\label{intro:tilt}
(Z, \aA) \leadsto
(Z_t^{\dag}, \aA^{\dag}), \ \  t>0,
\end{align}
where $\aA^{\dag}$ is  a tilt of $\aA$, and 
$Z_t^{\dag}$ is defined by $-i Z+t\Delta_I$. 
Applying the tilting process (\ref{intro:tilt})
twice starting from the classical slope stability,
we get the construction (\ref{intro:sigma})
in~\cite{BMT}. 
We show that, in some sense, the above 
properties (i), (ii)
are preserved under the operation (\ref{intro:tilt}). 
This yields the desired properties for the stability condition
 (\ref{intro:sigma}), and so proves Theorem~\ref{thm:intro:stack}. 
Finally, the quasi-properness of 
the moduli stack 
follows from the valuative criterion proved by 
Abramovich-Polishchuk~\cite{AP}.

\subsection{Donaldson-Thomas invariants}
Suppose that $X$ is a Calabi-Yau 3-fold 
satisfying Conjecture~\ref{intro:BMT}, for example an
A-type Calabi-Yau 3-fold (see \cite{BMS}). 
We use Theorem~\ref{thm:intro:stack} 
to define Donaldson-Thomas invariants 
counting Bridgeland semistable objects 
on such Calabi-Yau 3-folds. 
Following the construction of 
generalized DT invariants~\cite{JS}, \cite{K-S}
counting semistable sheaves, 
we construct a map
\begin{align}\notag
\DT_{\ast}(v) \colon \Stab_{\omega, B}^{\circ}(X) \to \mathbb{Q},
\end{align}
for each $v\in H^{\ast}(X, \mathbb{Q})$.
The invariant $\DT_{\sigma}(v)$ counts 
$\sigma$-semistable objects in $D^b \Coh(X)$
with Chern character $v$. 
In particular, the invariant
\begin{align}\label{intro:DTwB}
\DT_{\omega, B}(v) \cneq \DT_{\sigma_{\omega, B}}(v)
\end{align}
counts $Z_{\omega, B}$-semistable objects in $\aA_{\omega, B}$
that are certain three term complexes in the derived category. 
In a forthcoming paper, we will 
pursue the wall-crossing formula 
relating the invariants (\ref{intro:DTwB}) with 
the original DT invariants counting 
semistable sheaves, and show that they are invariant under 
 the deformations of the complex structure on $X$.

\subsection{Relation to the existing works}
As we mentioned before, Conjecture~\ref{conj:intro}~(ii)
for surfaces essentially follows from~\cite{Tst3},
and Theorem~\ref{thm:intro:stack}
is a 3-fold generalization. 
Moreover, the arguments in this paper contain
several improvements toward Conjecture~\ref{conj:intro} (ii). 
One of the key points of our approach is 
the formulation of the BG type inequality 
in the form (\ref{intro:BG}),
and to show that several properties are 
inherited from the tilting process (\ref{intro:tilt}). 
This generalized tilting argument would be 
convenient for the future study of Conjecture~\ref{conj:intro} (ii)
for higher dimensional varieties, 
which may require further tilting. 

Regarding the moduli problems involving objects in the derived category of 3-folds, 
the second author constructed moduli spaces of limit 
stable objects~\cite{Tolim}
on Calabi-Yau 3-folds. These are special cases of 
Bayer's polynomial stability conditions~\cite{Bay}.
Also they can be interpreted as limiting 
degenerations of Bridgeland stability conditions, 
and the moduli spaces constructed in~\cite{Tolim} may appear
as moduli stacks in Theorem~\ref{thm:intro:stack} 
at points which are sufficiently close 
to the so called \textit{large volume limit}. 
Also there exist works by Jason Lo~\cite{Lo1}, \cite{Lo2}
constructing moduli spaces of 
certain polynomial semistable objects called 
PT semistable objects, which 
may appear near the large volume limit as well. 
Similarly, the DT type invariants
constructed in~\cite{PT},~\cite{Tolim}
also may coincide with the  
invariants (\ref{intro:DTwB})
near the large volume limit.

\subsection{Plan of the paper}
In Section~\ref{sec:veryweak}, 
we introduce the notion of very weak stability conditions 
on triangulated categories, their BG type inequality and 
the associated tilting. The results of this section contain 
several generalizations of known results for tilt stability. 
In Section~\ref{sec:BGconj}, we interpret tilt stability and 
the associated double tilting in the framework of Section~\ref{sec:veryweak}, 
and proves an equivalent form of Conjecture~\ref{intro:BMT}
in Theorem~\ref{intro:equiv}. 
In Section~\ref{sec:moduli}, we prove Theorem~\ref{thm:intro:stack}
by showing general results on 
the generic flatness and the boundedness under the tilting (\ref{intro:tilt}). 
In Section~\ref{sec:DT}, 
using the result of Section~\ref{sec:moduli}, 
we define Donaldson-Thomas invariants 
counting Bridgeland semistable objects on Calabi-Yau 3-folds
satisfying Conjecture~\ref{intro:BMT}. 

\subsection{Acknowledgment}
The authors would like to thank the organizers of the 
workshop ``Geometry From Stability Conditions''
held on Feb. 16-20, 2015, at the University of Warwick. 
We are especially grateful to the referee for pointing out several 
errors and also for giving useful comments that led to a substantial improvement of this paper.
This work is supported by World Premier 
International Research Center Initiative
(WPI initiative), MEXT, Japan. The second author
 is also supported by Grant-in Aid
for Scientific Research grant (No.~26287002)
from the Ministry of Education, Culture,
Sports, Science and Technology, Japan.

\subsection{Notations and conventions}
Throughout this paper, all the varieties are defined over 
$\mathbb{C}$. 
For a variety $X$, by $\Coh(X)$ we denote  the category of 
coherent sheaves on $X$, and 
 $\Coh_{\le d}(X)$  denote its subcategory 
of coherent sheaves whose supports have 
dimension less than or equal to $d$. 
For simplicity, we write $\Coh_{\le 0}(X)$ as $\Coh_0(X)$. 
When  $\aA$ is the heart of a bounded t-structure  on 
a triangulated category $\dD$,  by 
$\hH_{\aA}^i(\ast)$ we denote the corresponding  $i$-th cohomology functor. 
When $\aA=\Coh(X)$ and $\dD = D^b \Coh(X)$, we simply write $\hH^i(\ast)$ for $\hH_{\Coh(X)}^i(\ast)$. 
For a set of objects $\sS \subset \dD$,  by 
$\langle \sS \rangle \subset \dD$ we denote
its extension closure, that is the smallest extension closed subcategory 
of $\dD$ which contains $\sS$.
We denote the upper half plane $\{z \in \mathbb{C} : \Imm z >0\}$ by $\mathbb{H}$. 

\section{Tilting via very weak stability conditions}\label{sec:veryweak}
In this section, we develop 
general arguments of very weak stability conditions, 
which are the variants of weak stability of~\cite{Tcurve1}
introduced in~\cite[Definition~B.1]{BMS}.
\subsection{Very weak stability conditions}\label{subsec:vw}
Let $\dD$ be a triangulated category, and 
$K(\dD)$ its Grothendieck group. 
We fix a finitely generated free abelian 
group $\Gamma$ and a group homomorphism
\begin{align*}
\cl \colon K(\dD) \to \Gamma. 
\end{align*}
We first give the following definition: 
\begin{defi}\label{defi:veryweak}
A very weak pre-stability condition 
on $\dD$ is a pair $(Z, \aA)$, 
where $\aA$ is the heart of a bounded t-structure on $\dD$, 
and $Z \colon \Gamma \to \mathbb{C}$ is a group homomorphism 
satisfying the following conditions: 
\begin{enumerate}
\item For any $E \in \aA$, we have
\begin{align*}
Z(E) \in \mathbb{H} \cup \mathbb{R}_{\le 0}.
\end{align*}
Here $Z(E) \cneq Z(\cl(E))$ and $\mathbb{H}$ is the 
upper half plane. 

\item Let 
\begin{align*}
\mu \cneq -\frac{\Ree Z}{\Imm Z}
\colon \aA \to \mathbb{R} \cup \{\infty\}
\end{align*} 
be the associated slope function. 
Here we set $\mu(E)=\infty$ if $\Imm Z(E)=0$. 
Then $\mu$ satisfies the Harder-Narasimhan (HN for short) property.  
\end{enumerate}
\end{defi}

We say that 
$E \in \aA$ is $\mu$-(semi)stable 
if for any non-zero subobject $F \subset E$
in $\aA$, 
we have the inequality
\begin{align*}
\mu(F) <(\le) \, \mu(E/F). 
\end{align*}
The HN filtration of an object $E \in \aA$ is 
a chain of subobjects 
\begin{align*}
0=E_0 \subset E_1 \subset \cdots \subset E_n=E
\end{align*}
 in $\aA$ such that each $F_i=E_i/E_{i-1}$ is 
$\mu$-semistable with 
$\mu(F_i)>\mu(F_{i+1})$. 
If such HN filtrations exist for all objects in $\aA$,
we say that $\mu$ satisfies the HN property.
\begin{rmk}
The definition of the $\mu$-semistable objects
is equivalent to the usual definition 
of slope semistability: 
$E \in \aA$ is $\mu$-semistable if and only if for any 
non-zero $F \subset E$ in $\aA$, we have
$\mu(F) \le \mu(E)$. However, the 
$\mu$-stable objects are different 
from those defined by 
the inequality
$\mu(F) <\mu(E)$.
\end{rmk}	

For a given a very weak pre-stability condition $(Z, \aA)$, 
we define its slicing on $\dD$ (see \cite[Definition~3.3]{Brs1})
\begin{align*}
\{\pP(\phi)\}_{\phi \in \mathbb{R}}, \
\pP(\phi) \subset \dD
\end{align*}
as in the case of Bridgeland 
stability conditions (see \cite[Proposition~5.3]{Brs1}). 
Namely, for $0<\phi \le 1$, 
the category $\pP(\phi)$ is defined to 
be 
\begin{align*}
\pP(\phi) =\{ 
E \in \aA : 
E \mbox{ is } \mu \mbox{-semistable with }
\mu(E)=-1/\tan (\pi \phi)\} \cup \{0\}.
\end{align*} 
Here we set $-1/\tan \pi =\infty$. 
The other subcategories are defined by setting
\begin{align*}
\pP(\phi+1)=\pP(\phi)[1].
\end{align*}
For an interval $I \subset \mathbb{R}$, 
we define $\pP(I)$ to be the smallest extension 
closed subcategory of $\dD$ which contains 
$\pP(\phi)$ for each $\phi \in I$. 

Note that  
the category $\pP(1)$ contains the 
following category 
\begin{align}\label{def:C}
\cC \cneq \{E \in \aA : Z(E)=0\}. 
\end{align}
It is easy to check that $\cC$ 
is closed under subobjects and quotients in $\aA$. 
In particular, $\cC$ is an 
abelian subcategory of 
$\aA$. 
We say that
 $(Z, \aA)$ is a \textit{pre-stability 
condition} if $\cC=\{0\}$. 
We define the subgroup $\Gamma_0 \subset \Gamma$
by 
\begin{align*}
\Gamma_0 \cneq [\cl(\cC)] \subset \Gamma. 
\end{align*}
Here for a subset $S \subset \Gamma$, 
by $[S]$ we denote  the saturation of the 
subgroup of $\Gamma$
generated by $S$. 
Note that $E \in \aA$ satisfies $\cl(E) \in \Gamma_0$
if and only if $E \in \cC$. 
Also $Z$ descends to the
group homomorphism 
\begin{align*}
\overline{Z} \colon 
\Gamma/\Gamma_0 \to \mathbb{C}.
\end{align*} 
We define the following analogue of support property 
introduced by Kontsevich-Soibelman~\cite{K-S}
for Bridgeland stability conditions: 
\begin{defi}\label{defi:support}
A very weak pre-stability condition 
$(Z, \aA)$ 
is a very weak stability condition if 
it satisfies
the support property:
there is a quadratic form 
$Q$ on $\Gamma/\Gamma_0$ 
satisfying  
\begin{enumerate}
\item $Q(E) \ge 0$ for any 
$\mu$-semistable object
$E \in \aA$, and 

\item $Q|_{\Ker \overline{Z}}$ is negative definite on 
$\Gamma/\Gamma_0$. 
\end{enumerate}
\end{defi}
For $v\in \Gamma$, 
we denote by $\overline{v}$ its image 
in $\Gamma/\Gamma_0$, 
and $\lVert \ast \rVert$ is a fixed norm on 
$(\Gamma/\Gamma_0) \otimes_{\mathbb{Z}} \mathbb{R}$. 
Similarly to the Bridgeland stability, 
the support property is also interpreted 
in the following way: 
\begin{lem}\label{rmk:support}
Let $(Z, \aA)$ 
be a very weak pre-stability condition. 
Then it satisfies the support property 
if and only if the following holds:  
\begin{align}\label{C:support}
C\cneq 
\sup
\left\{\frac{\lVert  \overline{\cl(E)} \rVert}{\lvert Z(E) \rvert}
: E \mbox{ is } \mu \mbox{-semistable with }
E \notin \cC \right\} < \infty. 
\end{align}
\end{lem}
\begin{proof}
If $(Z, \aA)$ is a pre-stability condition, then 
the lemma is stated in~\cite{K-S}
and the precise proof is 
available in~\cite[Lemma~A.4]{BMS}.
The same proof also applies for very weak pre-stability conditions. 
 Indeed if (\ref{C:support}) holds, 
then the quadratic form $Q$ in Definition~\ref{defi:support} is given by 
\begin{align}\label{quad:QQ}
Q(v)=C\lvert \overline{Z}(v) \rvert^2 - \lVert v \rVert^2
\end{align}
for $v \in \Gamma/\Gamma_0$. 
The proof of the converse is also the same as in~\cite[Lemma~A.4]{BMS}. 
\end{proof}

A very weak stability condition $\sigma=(Z, \aA)$
is called a \textit{stability condition}
if it is also a pre-stability condition, i.e. 
$\cC=\{0\}$.  
This notion coincides with the notion of 
Bridgeland stability conditions~\cite{Brs1}
satisfying the support property. 
In this case, we also call $\mu$-semistable objects
as $Z$-semistable objects, or as $\sigma$-semistable objects. 

\begin{exam}
When $\aA$ is the heart of a bounded t-structure on an arbitrary triangulated category $\dD$, 
by  definition,  
trivial pair $(Z=0, \aA)$ defines a  very weak stability condition.
\end{exam}
\begin{exam}
Let $A$ be a finite dimensional $\mathbb{C}$-algebra
and $\aA=\modu A$ the category of finitely generated 
right $A$-modules. 
Then there is a finite number of simple 
objects 
\begin{align*}
S_1, S_2, \cdots, S_k \in \aA
\end{align*}
such that $K(\aA)$ is freely 
generated by $[S_i]$ for $1\le i\le k$. 
We set $\Gamma=K(\aA)$ and $\cl=\mathrm{\id}$.
Then for 
any group homomorphism 
$Z \colon \Gamma \to \mathbb{C}$
with $Z([S_i]) \in \mathbb{H} \cup \mathbb{R}_{\le 0}$, 
it is easy to show that 
the pair $(Z, \aA)$ is a weak stability condition on 
$D^b(\aA)$. 
The subcategory $\cC \subset \aA$ is generated by 
$S_i$ with $Z(S_i)=0$. 
In fact, one can easily check that (\ref{C:support})
holds, and the quadratic form $Q$ is given 
as (\ref{quad:QQ}). 
\end{exam}
Let $\mathrm{Stab}_{\Gamma}^{\rm{vw}}(\dD)$ be the set of 
very weak stability conditions on $\dD$
with respect to the data $(\Gamma, \cl)$, and $\mathrm{Slice}(\dD)$ the 
set of slicing on $\dD$
with its topology introduced in~\cite[Section~6]{Brs1}. 
The set $\Stab_{\Gamma}^{\rm{vw}}(\dD)$ has a topology 
induced by the inclusion
\begin{align*}
\Stab_{\Gamma}^{\rm{vw}}(\dD) \subset 
\Hom(\Gamma, \mathbb{C}) \times \Slice(\dD). 
\end{align*}
Let $\Stab_{\Gamma}(\dD) \subset \Stab_{\Gamma}^{\rm{vw}}(\dD)$
be the subset consists of stability conditions. 
By~\cite[Theorem~1.2]{Brs1}, 
the map 
\begin{align*}
\Stab_{\Gamma}(\dD) \to \Hom(\Gamma, \mathbb{C})
\end{align*}
sending $(Z, \aA)$ to $Z$ is a 
local 
homeomorphism.

\subsection{Tilting of a very weak stability condition}
Following the idea to construct Bridgeland stability conditions on 
surfaces~\cite{Brs2}, \cite{AB}, 
and tilt stabilities on 3-folds~\cite{BMT}, 
we construct a `tilting' of 
a very weak stability condition $(Z, \aA)$. 
The additional required data is the 
following generalization of  
Bogomolov-Gieseker (BG)
inequality: 
\begin{defi}\label{defi:BG}
We say that $(Z, \aA)$ satisfies the 
BG inequality if 
there exist linear maps
\begin{align*}
\Delta_{R}, \Delta_{I} \colon 
\Gamma \to \mathbb{Q}
\end{align*}
satisfying the following conditions: 
\begin{enumerate}
\item For any $\mu$-semistable $E \in \aA$, we 
have the inequality
\begin{align}\label{ass1}
\Delta(E) \cneq 
\Ree Z(E) \Delta_{R}(E) + \Imm Z(E) \Delta_I(E) \ge 0. 
\end{align}
\item $\Delta_R|_{\cC}=0$, 
$\Delta_I|_{\cC} \le 0$ and $\Delta_I|_{\cC} \not\equiv  0$.  
\end{enumerate}
\end{defi}
As we will see in Subsection~\ref{subsec:tilt}, 
the inequality $\Delta(E) \ge 0$ is a generalization of 
the classical Bogomolov-Gieseker inequality 
for torsion free semistable sheaves on 
smooth projective  varieties. 
The inequality $\Delta(E) \ge 0$ 
also appears in an algebraic situation
as follows: 
\begin{exam}
Let $Q_l$ be the quiver with 
two vertex $\{1, 2\}$, and $l$-arrows 
from $1$ to $2$ (see the following picture for $l=2$)
\begin{align*}
Q_2 \colon 
\xymatrix{
1 
\ar@/_/[r] \ar@/^/[r]
 & 2 
}
\end{align*}
Let 
$\aA$ be the
category of finite dimensional 
$Q_l$-representations, and 
set $\dD=D^b(\aA)$. 
The group $\Gamma=K(\dD)$ is generated by 
$S_1$ and $S_2$, 
where $S_i \in \aA$ is the simple object
corresponding to the vertex $i$. 
We set the group homomorphism $Z \colon \Gamma \to \mathbb{C}$
by $Z(S_1)=0$ and $Z(S_2)=i$. 
Then 
$(Z, \aA)$ is a very weak stability condition with 
$\cC=\langle S_1 \rangle$, 
and 
$E \in \aA \setminus \cC$ is $\mu$-semistable 
if and only if
$\Hom(S_1, E)=0$. 
If we write $[E]=v_1[S_1]+v_2[S_2]$
in $\Gamma$, 
the condition $\Hom(S_1, E)=0$
implies $v_1 \le l v_2$. 
By setting 
\begin{align*}
(\Delta_R, \Delta_I)
=(0, lv_2-v_1)
\end{align*}
the data $(Z, \aA)$ satisfies the 
BG inequality. 
\end{exam}
\begin{rmk}
The condition $\Delta_R|_{\cC}=0$ naturally follows from 
$\Delta_I|_{\cC} \le 0$. 
Indeed for $s\in \mathbb{Q}$, 
the pair $(Z_s=Z+s \Imm Z, \aA)$
is also a very weak 
stability condition, whose semistable objects
coincide with $\mu$-semistable objects. 
We require that the BG inequality for $(Z_s, \aA)$ is 
equivalent to that of $(Z, \aA)$, i.e.
\begin{align*}
\Delta=\Ree Z_s \Delta_R + \Imm Z_s(\Delta_I-s\Delta_R) \ge 0
\end{align*}
is the BG inequality for $(Z_s, \aA)$. 
It requires $(\Delta_I-s\Delta_R)|_{\cC} \le 0$
for any $s\in \mathbb{Q}$, which implies 
$\Delta_R|_{\cC}=0$. 
\end{rmk}

Suppose that 
$(Z, \aA)$ satisfies the BG inequality, 
and fix the quadratic form $\Delta$ as above. 
For $t\in \mathbb{R}_{\ge 0}$, 
we define 
the group homomorphism 
$Z_t^{\dag} \colon \Gamma \to \mathbb{C}$ to be 
\begin{align*}
Z^{\dag}_t &\cneq -i Z+t\Delta_I =\left(\Imm Z + t\Delta_I\right) -\Ree Z \cdot i.
\end{align*}
We also define the heart 
$\aA^{\dag} \subset \dD$ to be 
\begin{align*}
\aA^{\dag} &\cneq 
\pP\left( \left({1}/{2}, \, {3}/{2} \right] \right).
\end{align*}

\begin{rmk}\label{rmk:tilt}
The heart $\aA^{\dag}$ is also described in the following way. 
Let $(\tT, \fF)$ be the pair of subcategories of $\aA$ defined by
\begin{align*}
&\tT \cneq \langle \mu\mbox{-semistable } E \in \aA \mbox{ with } 
\mu(E)>0  \rangle,  \\
&\fF \cneq \langle \mu\mbox{-semistable } E \in \aA \mbox{ with }
 \mu(E)\le 0 \rangle. 
\end{align*}
Then $(\tT, \fF)$ is a torsion pair as in \cite{HRS}, and 
$\aA^{\dag}$ coincides with the associated tilt:
\begin{align*}
\aA^{\dag}=\langle \fF[1], \tT \rangle. 
\end{align*}
In particular, any object $E \in \aA^{\dag}$ fits into
the exact sequence in $\aA^{\dag}$
\begin{align*}
0 \to \hH_{\aA}^{-1}(E)[1] \to E \to \hH_{\aA}^0(E) \to 0. 
\end{align*}
\end{rmk}
We also set 
\begin{align}\label{Cdag0}
\cC^{\dag} \cneq \{ F \in \cC : \Delta_I(F)=0\}.
\end{align}
We have the following lemma: 
\begin{lem}\label{lem:tilt1}
For any $E \in \aA^{\dag}$, we have 
$Z^{\dag}_t(E) \in \mathbb{H} \cup \mathbb{R}_{\le 0}$. 
Moreover, we have
\begin{align}\label{Cdag}
\{E \in \aA^{\dag} : Z_t^{\dag}(E)=0\}
=\left\{\begin{array}{cc}
\cC, & t=0, \\
\cC^{\dag}, & t>0.
\end{array}  \right. 
\end{align}
\end{lem}
\begin{proof}
By the construction of $\aA^{\dag}$, 
we have $\Imm Z^{\dag}_t(E)=-\Ree Z(E) \ge 0$
for any $E \in \aA^{\dag}$. 
Suppose that $\Imm Z^{\dag}_t(E)=0$. 
We have the short exact sequence
\begin{align*}
0 \to U[1] \to E \to F \to 0
\end{align*}
 in $\aA^{\dag}$, where $U=\hH_{\aA}^{-1}(E)$
and $F=\hH_{\aA}^0(E)$. 
We have
 $\Imm Z^{\dag}_t(U)=\Imm Z^{\dag}_t(F)=0$, and 
so $U \in \pP(1/2)$ and $F \in \cC$.  
We have
\begin{align*}
\Ree Z^{\dag}_t(E)=-\Imm Z(U)-t\Delta_I(U) + t\Delta_I(F). 
\end{align*}
If $U\neq 0$ then 
$\Imm Z(U)>0$, and (\ref{ass1}) implies 
$\Delta_I(U) \ge 0$. 
Combined with $\Delta_I(F) \le 0$ as $F \in \cC$, 
we obtain $\Ree Z^{\dag}_t(E) <0$. 
If $U=0$ then $E \in \cC$, and 
$\Ree Z^{\dag}(E)=0$ if and only if 
$t\Delta_I(E)=0$. 
Therefore, (\ref{Cdag}) also holds. 
\end{proof}

Similarly to $\mu$, we define the 
slope function $\mu^{\dag}_t \colon \aA^{\dag} \to 
\mathbb{R} \cup \{\infty\}$
to be 
\begin{align*}
\mu_t^{\dag} \cneq 
-\frac{\Ree Z^{\dag}_t}{\Imm Z_t^{\dag}} 
=\frac{\Imm Z +t \Delta_I}{\Ree Z}. 
\end{align*}
Here we set $\mu_t^{\dag}(E)=\infty$ if $\Ree Z(E)=0$. 

\begin{aside}
\rm
The following discussion highlights some potential applications of our general framework.
However, it is not directly relevant to our main purposes in this paper. 

An object of an abelian category is called a \textit{minimal object} when it has no
proper subobjects or equivalently no nontrivial quotients in the category. 
For example skyscraper sheaves of closed points are the only minimal objects of the abelian category of coherent sheaves on a scheme. 
This is a very useful algebraic notion and such minimal objects are preserved under an equivalence of abelian categories.
In our setting one can prove that (see Lemma~\ref{lem:t=0}~(iv) for
the notation and a similar argument)
$$
\langle \pP(1/2)[1], \cC \rangle = \pP^\dag_0(1) \subset  \aA^\dag.
$$
Therefore, it is an easy exercise to check that if $E \in \aA$ is 
\begin{itemize}
\item $\mu$-stable with
\item $\mu(E) = 0$, and 
\item $\Ext^{1}_{\aA}(\cC, E) =0$,
\end{itemize}
then $E[1] \in \aA^\dag$ has no proper subobjects in $\aA^\dag$. 
That is, $E[1] \in \aA^\dag$  is a minimal object.
The classes of minimal objects in  \cite[Proposition 2.2]{HuAb} and 
\cite[Lemma 2.3]{MaPi1} can be realized as examples of this result in
 the setup that we will introduce in Section \ref{sec:BGconj} for varieties.
Those classes of minimal objects played a crucial role in Bridgeland stability for
surfaces and 3-folds \cite{HuAb}, \cite{MaPi1}, \cite{MaPi2}, \cite{Piya}.
\end{aside}

\subsection{Harder-Narasimhan property after tilting}
In order to proceed further, we
need to show the HN property and 
the support property of $(Z_t^{\dag}, \aA^{\dag})$. 
For this purpose, we  
introduce the following 
technical condition.  
\begin{defi}\label{def:good}
Let $(Z, \aA)$ be a very weak stability 
condition satisfying a BG inequality. 
We say that $(Z, \aA)$ is good if the following 
conditions are satisfied: 
\begin{enumerate}
\item There exist constants $\zeta_i \in \mathbb{R}_{>0}$, $i=1,2$ such that
 $$
 Z(\Gamma) \subset \left(\zeta_1 \mathbb{Q}\right)+\left(\zeta_2 \mathbb{Q}\right)i.
 $$  
\item $\aA$ is a noetherian abelian category. 
\item 
We have $[\cl(\cC^{\dag})]=\Gamma_0^{\dag} \cneq \Gamma_0 
\cap \Ker(\Delta_I)$. 
\item For any $U \in \pP(\phi)$
with $\phi \in (0, 1)$, there is 
an injection $U \hookrightarrow \widehat{U}$ in $\aA$
such that
$\Hom(\cC^{\dag}, \widehat{U}[1])=0$. 
\end{enumerate}
\end{defi}
\begin{rmk}\label{rmk:good}
When $(Z, \aA)$ is a stability condition, 
it is good if and only if 
the condition (i) holds. 
(The condition (ii) follows from (i)
by~\cite[Proposition~5.0.1]{AP}.)
If $\zeta_i=1$, such 
a stability condition was called algebraic in~\cite{Tst3}. 
\end{rmk}
\begin{lem}\label{lem:terminate}
Suppose that a very weak stability 
condition $(Z, \aA)$ is good.  
Then 
for any $U \in \pP((0, 1))$,
 there is no infinite sequence
in $\aA$
\begin{align*}
U=U_1 \subset U_2 \subset \cdots \subset U_{i-1} \subset U_i \subset \cdots
\end{align*}
such that 
$U_i \in \pP((0, 1))$ and 
$U_i/U_{i-1} \in \cC$ for all $i$. 
\end{lem}
\begin{proof}
Suppose that such a sequence exists.  
Let $U \twoheadrightarrow Q$ be 
the minimal destabilizing quotient with respect to the
$\mu$-stability, and 
$P$ the kernel of $U \twoheadrightarrow Q$. 
Let $Q_i'$ be the cokernel of
the composition 
$P \hookrightarrow U \hookrightarrow U_i$, 
and define $Q_i$ to be the quotient of $Q_i'$ by 
its maximal subobject $T_i \subset Q_i'$ in $\aA$
with $T_i \in \cC$. 
Then we obtain the sequence in $\aA$
\begin{align}\notag
Q_1 \cneq Q \subset Q_2 \subset Q_3 \subset \cdots \subset
Q_{i-1} \subset Q_i \subset \cdots
\end{align}
such that $F_i \cneq Q_i/Q_{i-1} \in \cC$ for all
$i$. 
By the equality
\begin{align*}
\Delta_I(Q_i)=\Delta_I(Q_{i-1})+\Delta_I(F_i)
\end{align*}
and $\Delta_I(F_i) \le 0$, 
we have $\Delta_I(Q_i) \le \Delta_I(Q_{i-1})$. 
On the other hand, 
$Q_i$ is $\mu$-semistable by Sublemma~\ref{lem:prepare} below,
and so the BG inequality gives 
 $\Delta(Q_i) \ge 0$.
Combined with 
$Z(Q_i)=Z(Q)$
and $\Delta_R(Q_i)=\Delta_R(Q)$, we have
\begin{align*}
\Ree Z(Q) \Delta_R(Q)+\Imm Z(Q) \Delta_I(Q_i) \ge 0. 
\end{align*}
Since $\Imm Z(U)>0$, we have $\Imm Z(Q)>0$, 
and so $\Delta_I(Q_i)$ is bounded below. 
Therefore, we may assume that $\Delta_I(Q_i)$ is constant, 
which implies that $F_i \in \cC^{\dag}$. 
Now we take $Q \subset \widehat{Q}$ as in 
Definition~\ref{def:good} (iv). 
The condition $\Hom(\cC^{\dag}, \widehat{Q}[1])=0$ 
and $F_i \in \cC^{\dag}$ implies that 
\begin{align*}
\cdots \subset Q_{i-1} \subset Q_i \subset \cdots \subset \widehat{Q}.
\end{align*}
Since $\aA$ is noetherian, the above sequence must terminate. 
The result now follows by the induction on the number of HN 
factors of $U$. 
\end{proof}

\begin{sublem}\label{lem:prepare}
For $E \in \pP(\phi)$ with $\phi \in (0, 1]$, 
let $0 \to E \to E' \to F \to 0$
be an exact sequence in $\aA$ with 
$F \in \cC$ and 
$\Hom(\cC, E')=0$. Then $E' \in \pP(\phi)$. 
\end{sublem}
\begin{proof}
Let $0 \to P' \to E' \to Q' \to 0$ be an exact sequence in $\aA$
with $P', Q' \neq 0$, and let
$F' \subset F$ be the image of the composition 
$P' \hookrightarrow E'  \twoheadrightarrow F$. 
Since $\cC$ is closed under subobjects in $\aA$, we have
$F' \in \cC$. 
Let $P$ be the kernel of $P' \twoheadrightarrow F'$. 
Then $P \subset E$, and the
assumption $\Hom(\cC, E')=0$ implies that 
$P \neq 0$. 
As $\mu(P')=\mu(P)$ and $\mu(E')=\mu(E)$, 
the $\mu$-semistability of $E$ implies the $\mu$-semistability 
of $E'$.  
\end{proof}
\begin{lem}\label{lem:noether}
Suppose that a very weak stability condition 
$(Z, \aA)$
is good. Then 
$\aA^{\dag}$ is noetherian. 
\end{lem}
\begin{proof}
The result and the argument is similar to~\cite[Lemma~5.5.2]{BVer}. 
Suppose that there is an infinite sequence of 
surjections 
\begin{align}\label{surj}
E_1 \twoheadrightarrow E_2 \twoheadrightarrow \cdots
\twoheadrightarrow E_i \twoheadrightarrow E_{i+1} \twoheadrightarrow \cdots
\end{align}
in $\aA^{\dag}$.
Since $\Imm Z^{\dag}_t=-\Ree Z$ is 
discrete and non-negative on $\aA^{\dag}$, we may 
assume that 
$\Imm Z^{\dag}_t(E_i)$ is independent of $i$. 
Let us consider the sequence of short exact sequences in $\aA^{\dag}$
\begin{align*}
0 \to F_i \to E \cneq E_1 \to E_i \to 0. 
\end{align*}
Then we have the chain of surjections 
\begin{align*}
\hH_{\aA}^0(E) \twoheadrightarrow \hH_{\aA}^0(E_2)  \cdots \twoheadrightarrow\hH_{\aA}^0(E_i) \twoheadrightarrow 
\hH_{\aA}^0(E_{i+1}) \twoheadrightarrow \cdots
\end{align*}
in $\aA$. Hence,
we may assume that 
$\hH_{\aA}^0(E) \stackrel{\cong}{\to}
\hH_{\aA}^0(E_i)$
as $\aA$ is noetherian. 
Also we have the chain of 
inclusions 
\begin{align*}
\hH_{\aA}^{-1}(F_i) \subset \hH_{\aA}^{-1}(F_{i+1})
\subset \cdots \subset \hH_{\aA}^{-1}(E)
\end{align*} 
in $\aA$. So we may assume that $\hH_{\aA}^{-1}(F_i)$ is independent of 
$i$. By setting $V=\hH_{\aA}^{-1}(E)/\hH_{\aA}^{-1}(F_i)$, we 
have the exact sequence in $\aA$
\begin{align*}
0 \to V \to \hH_{\aA}^{-1}(E_i) \to \hH_{\aA}^0(F_i) \to 0. 
\end{align*}
Note that $\hH_{\aA}^0(F_i) \in \cC$ as
$\Ree Z(F_i)=0$. 
We write $U_i=\hH_{\aA}^{-1}(E_i) \in \pP((0, 1/2])$. 
Since $\cC$ is closed under subobjects and quotients 
 in $\aA$ and 
$\Hom(\cC, U_i)=0$, we have
the chain of inclusions in $\aA$
\begin{align*}
U \cneq U_1 \subset U_2 \subset \cdots \subset 
U_i \subset U_{i+1} \subset \cdots
\end{align*}
whose subquotients $U_i/U_{i-1}$ are contained in $\cC$. 
By Lemma~\ref{lem:terminate},
the above sequence terminates, and so 
the sequence (\ref{surj}) also terminates as required. 
\end{proof}
\begin{lem}
\label{lem:noetherian}
If $(Z, \aA)$ is good, then  
the data $(Z^{\dag}_t, \aA^{\dag})$ is a very 
weak pre-stability condition. 
\end{lem}
\begin{proof}
Since $\aA^{\dag}$ is noetherian by Lemma~\ref{lem:noether}, 
it is enough to show that
there is no infinite sequence of 
subobjects in $\aA^{\dag}$
\begin{align*}
\cdots \subset E_{i+1} \subset E_i \subset \cdots
\subset E_2 \subset E_1
\end{align*}
with $\mu^{\dag}_t(E_{i})>\mu^{\dag}_t(E_{i}/E_{i-1})$ (see \cite[Proposition~2.12]{Tcurve1}). 
Suppose that such a sequence exists. 
Since $\Imm Z^{\dag}_t$ values are non-negative and discrete 
on $\aA^{\dag}$, 
we may assume that 
$\Imm Z^{\dag}_t(E_{i-1})=\Imm Z^{\dag}_t(E_{i})$
for all $i$. 
Then $\mu^{\dag}_t(E_i/E_{i-1})=\infty$; this is the required contradiction. 
\end{proof}

\subsection{Support property after tilting}
Let $(Z, \aA)$ be a very weak stability condition which 
satisfies a BG inequality and is good. 
In this subsection, we show 
the support property of $(Z_t^{\dag}, \aA^{\dag})$. 
Let 
$\pP_t^{\dag}(\phi)$ 
be the slicing 
associated with $(Z_t^{\dag}, \aA^{\dag})$. 
We first 
investigate the category 
$\pP^{\dag}_t(\phi)$
when $t=0$.  
\begin{lem}\label{lem:t=0}
The category $\pP^{\dag}_0(\phi)$ is described
in the following way: 
\begin{enumerate}
\item If $0<\phi<1/2$, then
$\pP^{\dag}_0(\phi)=\pP(\phi+1/2)$. 
\item $\pP^{\dag}_0(1/2)$ consists of 
$E \in \pP(1)$ with $\Hom(\cC, E)=0$. 
\item If $1/2<\phi<1$, then 
$\pP^{\dag}_0(\phi)$ consists of objects
$E \in \langle \pP(\phi+1/2), \cC \rangle$
with $\Hom(\cC, E)=0$. 
\item $\pP^{\dag}_0(1)$ coincides with 
$\langle \pP(3/2), \cC \rangle$. 
\end{enumerate}
\end{lem}
\begin{proof}
We only prove the case when $1/2<\phi<1$. 
One can prove the other cases similarly. 
Let us take $E \in \pP^{\dag}_0(\phi)$
with $1/2<\phi<1$, i.e. $E\in \aA^{\dag}$ is  
$\mu_0^{\dag}$-semistable 
with $\mu_0^{\dag}(E)=-1/\tan (\pi \phi) >0$. 
We have the exact sequence
\begin{align*}
0 \to \hH_{\aA}^{-1}(E)[1] \to E \to \hH^0_{\aA}(E) \to 0
\end{align*}
in $\aA^{\dag}$. Let $F \subset \hH^0_{\aA}(E)$ be the 
maximum subobject in $\aA$ with $F \in \cC$, 
and let $T=\hH^0_{\aA}(E)/F$. 
If $T \neq 0$ then $\mu_0^{\dag}(T)< 0$, 
and this is not possible 
as $E$ is $\mu_0^{\dag}$-semistable. 
Hence, $\hH^0_{\aA}(E) \in \cC$. 
As $Z_0^{\dag}(\hH^0_{\aA}(E))=0$, by considering the HN 
filtration of $\hH_{\aA}^{-1}(E)$, one 
can easily check that $\hH_{\aA}^{-1}(E)$ is $\mu$-semistable. 
Hence, $E \in \langle \pP(\phi+1/2), \cC \rangle$
holds. 
Also the $\mu_0^{\dag}$-semistability of $E$
implies that $\Hom(\cC, E)=0$. 

Conversely, let us take an object
$E \in \langle \pP(\phi+1/2), \cC \rangle$
for $1/2<\phi<1$
satisfying $\Hom(\cC, E)=0$, 
and an exact sequence 
\begin{align*}
0 \to P \to E \to Q \to 0
\end{align*}
in $\aA^{\dag}$ with $P, Q \neq 0$. 
Note that $P\notin \cC$ as
$\Hom(\cC, E)=0$. 
By taking the long exact sequence of 
the cohomology functor $\hH_{\aA}^{\ast}(\ast)$, we get
 $\hH^{-1}_{\aA}(P) \in \pP((0, \phi-1/2])$. 
Since $\hH^0_{\aA}(P) \in \pP((1/2, 1])$
and $P\notin \cC$, 
we have the inequalities 
\begin{align*}
\mu_0^{\dag}(P) \le \mu_0^{\dag}(\hH_{\aA}^{-1}(P)[1]) \le 
\mu_0^{\dag}(E).
\end{align*}
Therefore, $E$ is $\mu_0^{\dag}$-semistable.  
\end{proof}

Let $Q$ be a quadratic form 
on $\Gamma/\Gamma_0$ 
which gives the support property of $(Z, \aA)$,
and 
$\Gamma_0^{\dag}=\Gamma_0 \cap \Ker(\Delta_I)$
the subgroup of $\Gamma_0$
considered in Definition~\ref{def:good}. 
Note that $Z_t^{\dag}$ descends to the group
homomorphism
\begin{align*}
\overline{Z}_t^{\dag} \colon 
\Gamma/\Gamma_0^{\dag} \to \mathbb{C}. 
\end{align*}
\begin{lem}\label{lem:quad}
There exists $k_0>0$ such that 
the quadratic form 
\begin{align*}
\Delta_{k, t} \cneq k Q +t \Delta
\end{align*}
on $\Gamma/\Gamma_0^{\dag}$ is 
negative definite on 
$\Ker (\overline{Z}_t^{\dag})$
for any $k \in (0, k_0)$ and $t>0$. 
\end{lem}
\begin{proof}
We first check that $\Delta_{k, t}$ is a 
quadratic form on $\Gamma/\Gamma_0^{\dag}$. 
It is enough to check that $\Delta$ descends to 
$\Gamma/\Gamma_0^{\dag}$, or equivalently 
for $\gamma_1 \in \Gamma_0^{\dag}$ and 
$\gamma_2 \in \Gamma$, we have 
\begin{align}\label{descend}
\Delta(\gamma_1+ \gamma_2)-\Delta(\gamma_1)-\Delta(\gamma_2)=0.
\end{align}
The LHS of (\ref{descend}) is calculated as
\begin{align*}
\Ree Z(\gamma_1)\Delta_R(\gamma_2)+\Ree Z(\gamma_2)\Delta_R(\gamma_1)+
\Imm Z(\gamma_1)\Delta_I(\gamma_2)+\Imm Z(\gamma_2)\Delta_I(\gamma_1). 
\end{align*} 
Since $\gamma_1 \in \Gamma_0^{\dag}$, we have 
$\Ree Z(\gamma_1)=\Imm Z(\gamma_1)=\Delta_I(\gamma_1)=0$. 
Also $\Delta_{R}|_{\cC}=0$ in Definition~\ref{defi:BG}
implies $\Delta_R(\gamma_1)=0$. Therefore, (\ref{descend}) holds. 

We take a non-zero 
$\gamma \in \Gamma/\Gamma_0^{\dag}$ such that 
$\overline{Z}_t^{\dag}(\gamma)=0$, i.e. 
$\Ree Z(\gamma)=0$ and $\Imm Z(\gamma)+ t\Delta_I(\gamma)=0$. 
So we have
\begin{align}\label{kQ}
\Delta_{k, t}(\gamma)=kQ(\gamma)-(\Imm \overline{Z}(\gamma))^2. 
\end{align}
Since $Q$ is negative definite on $\Ker(\overline{Z})$
in $\Gamma/\Gamma_0$, 
one can find $k_0>0$, which is independent of 
$t$ and $\gamma$, such that 
the RHS of (\ref{kQ}) is always negative
for any $k \in (0, k_0)$. 
\end{proof}
\begin{prop}\label{lem:Delkt}
For any $k \in (0, k_0)$, $t>0$
and a
$\mu^{\dag}_t$-semistable 
object $E \in \aA^{\dag}$, 
we have $\Delta_{k, t}(E) \ge 0$. 
In particular, we have  
$\Delta(E) \ge 0$. 
\end{prop}
\begin{proof}
Note that for any $E \in \aA^{\dag}$, 
we have $\Imm Z_t^{\dag}(E)=-\Ree Z(E) \ge 0$. 
We show the claim by the induction on 
$-\Ree Z(E)$. 
We define $R>0$ to be
\begin{align}\label{min}
R \cneq \mathrm{min}\{ -\Ree Z(E)>0 : 
E \in \aA^{\dag}\}. 
\end{align}
If $-\Ree Z(E)$ is either zero or $R$, then 
any $\mu_t^{\dag}$-semistable object $E \in \aA^{\dag}$ is 
also $\mu_t^{\dag}$-semistable for $0<t\ll 1$, 
and so it is $\mu_0^{\dag}$-semistable. 
By Lemma~\ref{lem:t=0}, 
the object $E$ satisfies 
$Q(E) \ge 0$ and $\Delta(E) \ge 0$, 
and so the inequality $\Delta_{k, t}(E) \ge 0$ holds. 
Suppose that $-\Ree Z(E)>R$. If $E$ is $\mu_t^{\dag}$-semistable 
for $0<t\ll 1$, then 
again Lemma~\ref{lem:t=0} shows that 
$\Delta_{k, t}(E) \ge 0$. 
Otherwise, there is $0<t'<t$ such that 
$E$ is $\mu_{t'}^{\dag}$-semistable, 
and an exact sequence 
\begin{align*}
0\to E_1 \to E \to E_2 \to 0
\end{align*}
with $\mu_{t'}^{\dag}(E_1)=\mu_{t'}^{\dag}(E_2)$. 
By the induction hypothesis, we have 
$\Delta_{k', t'}(E_i) \ge 0$
for any $k'\in (0, k_0)$
and $i=1, 2$. 
Let $\cC_{E, k', t} \subset \Gamma_{\mathbb{R}}$ be the 
cone defined by 
\begin{align*}
\cC_{E, k', t} \cneq 
\{ v \in \Gamma_{\mathbb{R}} : 
Z_{t}^{\dag}(v) \in \mathbb{R}_{\ge 0} Z_t^{\dag}(E), 
\Delta_{k, t}(v) \ge 0 \}. 
\end{align*}
By~\cite[Lemma~A.7]{BMS} the cone $\cC_{E, k', t}$ is convex,
and so $\Delta_{k', t'}(E) \ge 0$. 
By setting $k'=kt'/t$, we obtain the inequality
$\Delta_{k, t}(E)=t\Delta_{k', t'}(E)/t' \ge 0$. 
The last statement follows by taking $k \to 0^+$. 
\end{proof}
We have the following corollary:
\begin{cor}\label{cor:support}
Suppose that 
$(Z, \aA)$ is good. 
Then $(Z_t^{\dag}, \aA^{\dag})$ is a 
very weak stability condition such that the 
map defined by
\begin{align}\label{slice:cont}
\mathbb{R}_{>0} \to \Stab^{\rm{vw}}_{\Gamma}(\dD), \  t \mapsto (Z_t^{\dag}, \aA^{\dag})
\end{align}
is continuous. 
\end{cor}
\begin{proof}
By Lemma~\ref{lem:quad} and Proposition~\ref{lem:Delkt}, 
$(Z_t^{\dag}, \aA^{\dag})$ satisfies the support property. Hence,
it is a very weak stability condition. 
In order to show that the map (\ref{slice:cont}) is continuous,
it is enough to show that the map
\begin{align}\label{slice:cont2}
\mathbb{R}_{>0} \to \Slice(\dD), \  t \mapsto \{\pP_t^{\dag}(\phi)\}_{\phi \in \mathbb{R}}
\end{align}
is continuous. 
Let $K \subset \mathbb{R}_{>0}$ be a 
closed interval. 
Since the quadratic form $\Delta_{k, t}$ is continuous with respect to $t$, 
we can choose a quadratic form $\Delta_K$ which is independent of $t$
giving the support property of $(Z_t^{\dag}, \aA^{\dag})$ for any $t\in K$. 
In particular, for any 
$t, t' \in K$ and 
$E \in \pP_{t}^{\dag}(\phi)$ with $\phi \in (0, 1)$, 
we have the inequality (see~Remark~\ref{rmk:support})
\begin{align}\label{support:Z}
\left \lvert 1-\frac{Z_{t'}^{\dag}(E)}{Z_{t}^{\dag}(E)} \right \rvert 
<C \lvert t-t' \rvert,
\end{align}
where $C>0$ is a constant which is independent of $t, t'$ and $\phi$.
If necessary, we may replace $K$ by a smaller interval 
such that $C \lvert t-t' \rvert$ is  small enough, say less than $1/8$. 
Let $E \twoheadrightarrow F$ be a
destabilizing quotient of $E$
in $\aA^{\dag}$  with respect to the
 $\mu_{t'}^{\dag}$-stability. 
Applying (\ref{support:Z}) to $E$ and $F$, we have 
\begin{align*}
&\lvert \arg Z_t^{\dag}(E) - \arg Z_{t'}^{\dag}(E) \rvert <\pi \theta, \\ 
&\lvert \arg Z_t^{\dag}(F) - \arg Z_{t'}^{\dag}(F) \rvert <\pi \theta.
\end{align*}
Here $\theta \in [0, 1)$ is determined by 
$\sin (\pi \theta/2) =C\lvert t-t' \rvert$. 
On the other hand, we have 
$\arg Z_t^{\dag}(E) \le \arg Z_t^{\dag}(F)$
and $\arg Z_{t'}^{\dag}(E)>
 \arg Z_{t'}^{\dag}(F)$. 
Therefore,
\begin{align*}
\lvert \arg Z_{t}^{\dag}(E) - \arg Z_{t'}^{\dag}(F) \rvert <\pi \theta.
\end{align*}
Similar arguments show the 
same inequality holds for destabilizing 
subobjects of $E$. As a result, we obtain 
\begin{align*}
\pP_{t}^{\dag}(\phi) \subset \pP_{t'}^{\dag}((\phi-\theta, \phi+\theta)). 
\end{align*}
As $\theta \to 0$ when $t' \to t$, 
and $\pP_t^{\dag}(1)$ is independent of $t$, the map (\ref{slice:cont2})
is continuous. 
\end{proof}

The following corollary follows from Lemma~\ref{lem:tilt1} 
and Corollary~\ref{cor:support}. 
\begin{cor}
Suppose that $(Z, \aA)$ is good.
Then the 
following conditions are equivalent for $t>0$: 
\begin{enumerate}
\item $(Z_t^{\dag}, \aA^{\dag})$ is a stability condition. 
\item $Z_t^{\dag}|_{\pP(1)}$ is a stability condition on $\pP(1)$. 
\item $\Delta_I$ is strictly negative on $\cC \setminus \{0\}$. 
\end{enumerate}
\end{cor} 

\begin{rmk}
Even if $\cC^{\dag} \neq 0$, 
as $\cC^{\dag} \subsetneq \cC$, 
the very weak stability condition 
$(Z^{\dag}_t, \aA^{\dag})$ for $t>0$ is closer to 
a stability condition. 
\end{rmk}

Finally in this subsection, we give a speculative 
argument for the BG inequality of $(Z_t^{\dag}, \aA^{\dag})$. 
Let $E \in \aA^{\dag}$ be a $\mu_t^{\dag}$-semistable object. 
By Proposition~\ref{lem:Delkt},   $\Delta(E) \ge 0$.
Hence, we have
\begin{align}\label{Del:dag}
\Ree Z_t^{\dag}(E) \Delta_I(E)+
\Imm Z_t^{\dag}(E) (-\Delta_R(E)) \ge t \Delta_I(E)^2 \ge 0. 
\end{align} 
Moreover, $\Delta_I|_{\cC^{\dag}}=0$
and $\Delta_R|_{\cC^{\dag}}=0$ holds. 
This implies that, although the inequality (\ref{Del:dag})
is not a BG inequality for $(Z_t^{\dag}, \aA^{\dag})$, 
it is very close to it. 
We expect that a BG inequality for $(Z_t^{\dag}, \aA^{\dag})$
is obtained by adding 
some natural correction term 
$\nabla_I$ on $-\Delta_R$
satisfying $\nabla_I|_{\cC} \le (\not \equiv) 0$.
Namely, by setting
\begin{align}\label{DRIdag}
(\Delta_R^{\dag}, \Delta_I^{\dag})
=(\Delta_I, \nabla_I-\Delta_R),
\end{align}
we may obtain a BG inequality for 
$(Z_t^{\dag}, \aA^{\dag})$, i.e. 
for any $\mu_t^{\dag}$-semistable object 
$E \in \aA^{\dag}$, we may have
\begin{align}\label{Del:dag2}
\Ree Z_t^{\dag}(E) \Delta_R^{\dag}(E)+
\Imm Z_t^{\dag}(E) \Delta_I^{\dag}(E) \ge 0. 
\end{align}
In Subsection~\ref{subsec:tilt}, we will 
see that the above form of the BG inequality (\ref{Del:dag2})
exactly matches with 
the conjectural 
BG inequality for 
tilt semistable objects on 3-folds in~\cite{BMT}.  

\subsection{The limit $t\to 0^+$ }\label{subsec:to0}
We will carry the notation introduced in the previous subsection. 
The purpose of this subsection is to investigate
the $\mu_t^{\dag}$-semistable objects for $0<t\ll 1$. 
We first prepare the following lemma: 
\begin{lem}\label{lem:Delta}
Let $E \in \aA^{\dag}$ be  a $\mu^{\dag}_t$-semistable object
 with $\mu^{\dag}_t(E)<\infty$. Let 
$0 \to E_1 \to E \to E_2 \to 0$
be a short exact sequence in $\aA^{\dag}$ 
with $\mu_t^{\dag}(E_1)=\mu_t^{\dag}(E_2)$.
Let  $R_i \cneq \Ree Z(E_i)$ and $I_i \cneq \Imm Z(E_i)$.
Then 
\begin{align*}
\Delta(E) \ge \Delta(E_1)+ \Delta(E_2)+ 
\frac{R_1 R_2}{t} \left(\frac{I_1}{R_1}-\frac{I_2}{R_2} \right)^2. 
\end{align*}
\end{lem}
\begin{proof}
The condition $\mu_t^{\dag}(E_1)=\mu_t^{\dag}(E_2)$ is equivalent to 
\begin{align}\label{Del1}
\frac{I_1+t \Delta_I(E_1)}{R_1}=\frac{I_2+t \Delta_I(E_2)}{R_2}. 
\end{align}
Also since $\Imm Z^{\dag}_t(E_i)=-R_i>0$, 
the condition $\Delta(E_i)\ge 0$
in Proposition~\ref{lem:Delkt} is equivalent to 
\begin{align}\label{Del2}
\Delta_R(E_i) \le -\frac{I_i \Delta_I(E_i)}{R_i}. 
\end{align}
We have
\begin{align*}
&\Delta(E) -\Delta(E_1)-\Delta(E_2) \\
& \qquad =R_1 \Delta_R(E_2)+R_2 \Delta_R(E_1)
+I_1 \Delta_{I}(E_2)+I_2 \Delta_I(E_1) \\
& \qquad  \ge -R_1 \frac{I_2 \Delta_I(E_2)}{R_2}
-R_2 \frac{I_1 \Delta_I(E_1)}{R_1}
+I_1 \Delta_{I}(E_2)+I_2 \Delta_I(E_1) \\
& \qquad =R_1 R_2 \left(\frac{I_1}{R_1}-\frac{I_2}{R_2} \right) 
\left(\frac{\Delta_I(E_2)}{R_2}-\frac{\Delta_I(E_1)}{R_1} \right) \\
& \qquad =\frac{R_1 R_2}{t} \left(\frac{I_1}{R_1}-\frac{I_2}{R_2} \right)^2. 
\end{align*}
Here we have used (\ref{Del2})
for the first inequality, and (\ref{Del1})
for the last equality. 
Therefore, we obtain the desired result. 
\end{proof}
For
$v\in \Gamma$,  
we define 
$M_{t}^{\dag}(\overline{v})$
to be the set of 
isomorphism classes of $\mu_t^{\dag}$-semistable
objects
$E \in \aA^{\dag}$ with $\overline{\cl(E)} =\overline{v}$
in $\Gamma/\Gamma_0^{\dag}$. 
Note that, similarly to the existence of 
wall and chamber structures in the 
space of Bridgeland 
stability conditions (see~\cite[Proposition~9.3]{Brs2}), 
the result of Corollary~\ref{cor:support}
 implies 
the existence of locally finite 
set of points $W \subset \mathbb{R}_{>0}$ called 
\textit{walls}, such that $M_t^{\dag}(\overline{v})$
is constant on 
each connected component of $\mathbb{R}_{>0} \setminus W$. 
In the following lemma, we 
show that there is no point $t \in W$
for $0<t \ll 1$.
\begin{lem}\label{lem:wallemp}
There is $t_0>0$ such that 
$M_t^{\dag}(\overline{v})$ is constant for $0<t<t_0$. 
\end{lem}
\begin{proof}
We may assume that 
$\Imm Z_t^{\dag}(v)=-\Ree Z(v)$ is positive, 
and let $R \cneq \Ree Z(v)$. 
For $E \in M_t^{\dag}(\overline{v})$, suppose that 
there is an exact sequence 
\begin{align*}
0 \to E_1 \to E \to E_2 \to 0
\end{align*}
in 
$\aA^{\dag}$ such that 
$\mu_t^{\dag}(E_1)=\mu_t^{\dag}(E_2)$
and $\mu_{t-\varepsilon}^{\dag}(E_1)>\mu_{t-\varepsilon}^{\dag}(E_2)$
for $0<\varepsilon \ll 1$. 
Using the same notation in the proof of 
Lemma~\ref{lem:Delta}, 
these 
conditions 
imply that 
\begin{align}\label{Del:pos}
\frac{I_1}{R_1}-\frac{I_2}{R_2}>0. 
\end{align}
We take $r \in \mathbb{R}_{>0}$ 
so that $r \Imm Z$ is an integer valued function
on $\Gamma$. 
Then the LHS of (\ref{Del:pos}) is bigger than or equal to 
$1/rR_1 R_2$. 
Since $0<-R_i <-R$
with $R_1+R_2=R$,
and $\Delta(E_i) \ge 0$, 
 from Lemma~\ref{lem:Delta}
we have
\begin{align*}
\Delta(E) \ge \frac{1}{tr^2 R_1 R_2} 
\ge \frac{4}{tr^2 R^2}.
\end{align*}
By the proof of Lemma~\ref{lem:quad}, 
the quadratic form $\Delta$ descends to $\Gamma/\Gamma_0^{\dag}$, and so 
$\Delta(E)=\Delta(v)$. 
Therefore, by setting 
$t_0 \cneq \frac{4}{\Delta(v)r^2 R^2}$, 
the set of objects
$M_t^{\dag}(\overline{v})$ is constant for $0<t<t_0$. 
\end{proof}

We consider the restriction of $Z_t^{\dag}$ to 
$\pP^{\dag}_0(\phi)$. 
For $0<\phi<1$, the image of 
$Z_t^{\dag}|_{\pP^{\dag}_0(\phi)}$ is contained in the upper half plane, 
and $(Z_t^{\dag}|_{\pP^{\dag}_0(\phi)}, \pP^{\dag}_0(\phi))$ 
is a pre-stability condition. 
The restriction
of $\mu_t^{\dag}$ to $\pP^{\dag}_0(\phi)$ 
is given by 
\begin{align*}
\mu_t^{\dag}|_{\pP^{\dag}_0(\phi)}=\tan (\pi \phi)
+ t \frac{\Delta_I}{\Ree Z}. 
\end{align*}
Hence,  $\mu_t^{\dag}|_{\pP^{\dag}_0(\phi)}$-semistable objects on 
$\pP^{\dag}_0(\phi)$ do not depend on $t$, 
and coincide with the $\lambda|_{\pP^{\dag}_0(\phi)}$-semistable 
objects, where 
$\lambda$ is the slope function
\begin{align}\label{def:lambda}
\lambda \cneq \frac{\Delta_I}{\Ree Z}.
\end{align}
\begin{prop}\label{lem:tsmall}
Let $v \in \Gamma$ and let $t_0$ be the corresponding positive real number as in Lemma \ref{lem:wallemp}.
Then for $0<t<t_0$, the set $M_t^{\dag}(\overline{v})$
consists of isomorphism classes of $\lambda|_{\pP^{\dag}_0(\phi)}$-semistable
objects $E \in \pP^{\dag}_0(\phi)$ with 
$\overline{\cl(E)}=\overline{v}$.
Here $\phi$ is determined by 
$\mu_0^{\dag}(v)=-1/\tan (\pi \phi)$. 
\end{prop}
\begin{proof}
Let us take an
object $[E] \in M_t^{\dag}(\overline{v})$
for $0<t<t_0$.  
Then $E$ must be $\mu_0^{\dag}$-semistable, and so
it is an object in $\pP^{\dag}_0(\phi)$ for some $0<\phi \le 1$
satisfying $\mu_0^{\dag}(v)=-1/\tan (\pi \phi)$. 
Since $E$ is $\mu_t^{\dag}$-semistable, it must be
$\lambda|_{\pP^{\dag}_0(\phi)}$-semistable. 

Conversely, take a $\lambda|_{\pP^{\dag}_0(\phi)}$-semistable 
object $E \in \pP^{\dag}_0(\phi)$
with $\overline{\cl(E)}=\overline{v}$. 
Let us take a short exact sequence 
\begin{align*}
0 \to P \to E \to Q \to 0
\end{align*}
in $\aA^{\dag}$. 
Since $E$ is $\mu_0^{\dag}$-semistable, we 
have $\mu_0^{\dag}(P) \le \mu_0^{\dag}(Q)$. 
If $\mu_0^{\dag}(P)<\mu_0^{\dag}(Q)$ then 
$\mu_{t}^{\dag}(P)<\mu_t^{\dag}(Q)$ 
for $0<t<t_0$ as there are no walls on $(0, t_0)$. 
If $\mu_0^{\dag}(P)=\mu_0^{\dag}(Q)$ then 
$P, Q \in \pP^{\dag}_0(\phi)$ and the 
$\lambda|_{\pP_0^{\dag}(\phi)}$-semistability 
of $E$ implies 
$\mu_t^{\dag}(P) \le \mu_t^{\dag}(Q)$ for 
$0<t<t_0$. 
Therefore, we have 
$[E] \in M_t^{\dag}(\overline{v})$. 
\end{proof}

\section{Bogomolov-Gieseker inequality conjecture}\label{sec:BGconj}
In this section, we interpret the tilt stability 
and the double tilting construction in~\cite{BMT}
in terms of tilting 
in the general framework of very weak 
stability conditions in the previous section.
We also recall the BG inequality conjecture in~\cite{BMT}, and 
give two equivalent forms of it
generalizing 
\cite[Theorem~4.2, Theorem 5.4]{BMS} due to Bayer, Macr\'i and Stellari.
\subsection{The space of Bridgeland stability conditions}
Let $X$ be an $n$-dimensional smooth projective variety.
Let 
\begin{align}
\label{class:choice}
B\in \mathrm{NS}(X)_{\mathbb{Q}}, \text{ and } \omega \in \mathrm{NS}(X)_{\mathbb{R}} \text{ an ample class with } \omega^2 \text{ rational}.
\end{align}
That is $\omega = mH$ for some ample divisor class $H \in \mathrm{NS}(X)$ with $m^2 \in \mathbb{Q}_{>0}$.

Let $\ch^B(E) \cneq e^{-B}\ch(E)$ be the twisted Chern character of $E$ with respect to $B$. 
We set 
$v_j^B(E) \cneq \omega^{n-j}\ch_j^B(E)$, and 
\begin{align*}
v^B(E) \cneq (v_0^B(E), v_1^B(E), \cdots, v_n^B(E)).
\end{align*} 
Let 
$\Gamma_{\omega, B} \subset \mathbb{R}^{n+1}$
be the 
free abelian group of rank $(n+1)$ given by the image of the map
\begin{align}\label{def:vB}
v^{B} \colon K(X) \twoheadrightarrow \Gamma_{\omega, B}.
\end{align}

In what follows, we write an element of $\Gamma_{\omega, B}$
as $(v_0^B, v_1^B, \cdots, v_n^B)$. 
Applying the definitions in Subsection~\ref{subsec:vw}
to the following setting
\begin{align*}
\dD=D^b \Coh(X), \ \Gamma=\Gamma_{\omega, B}, \ 
\cl=v^B
\end{align*}
we have the space of very weak stability conditions 
\begin{align*}
\Stab_{\omega, B}^{\rm{vw}}(X) \cneq 
\Stab_{\Gamma_{\omega, B}}^{\rm{vw}}(D^b \Coh(X)).
\end{align*}
The subset $\Stab_{\omega, B}(X) \subset \Stab_{\omega, B}^{\rm{vw}}(X)$
of Bridgeland stability conditions 
admit the local homeomorphism by 
the forgetful map
\begin{align*}
\Stab_{\omega, B}(X) \to \Hom(\Gamma_{\omega, B}, \mathbb{C}), \ (Z,\aA) \mapsto Z
\end{align*}
as long as $\Stab_{\omega, B}(X)$ is non-empty.  
\subsection{Tilting via slope stability}\label{subsec:slope}
We interpret the 
classical slope stability on $\Coh(X)$
as a very weak stability condition. 
We set
$Z \colon \Gamma \to \mathbb{C}$ to be
\begin{align*}
Z(v)=-v_1^B + v_0^B i. 
\end{align*}
Here we have denoted $\Gamma=\Gamma_{\omega, B}$ for simplicity. 
It is easy to check that 
\begin{align}\label{ZCoh}
(Z, \Coh(X)) \in \Stab_{\omega, B}^{\rm{vw}}(X).
\end{align} 
Indeed, the 
associated slope function on $\Coh(X)$ is given by
the twisted slope
\begin{align}\label{tslope}
\mu_{\omega,B}(E) \cneq 
\frac{v_1^B(E)}{v_0^B(E)}. 
\end{align}
Here we set $\mu_{\omega, B}(E)=\infty$ when $E$ is 
a torsion sheaf. 
Since $\mu_{\omega, B}=\mu_{\omega}-(B\omega^2)/(\omega^3)$
for $\mu_{\omega} \cneq \mu_{0, \omega}$, 
 $\mu_{\omega, B}$-stability is independent of $B$.
Also  the existence of HN filtrations
with respect to the $\mu_{\omega}$-stability 
is well-known (see~\cite[Section~1]{Hu} for further details). 
Moreover, the trivial quadratic form $Q=0$
gives the support property of (\ref{ZCoh}). 
\begin{rmk}
If we take $\Gamma$ to be the image of the Chern character map 
in $H^{\ast}(X, \mathbb{Q})$ rather than $\Gamma_{\omega, B}$, 
the pair $(Z, \Coh(X))$ does not satisfy the support property 
when the Picard number is bigger than one. 
\end{rmk}
Suppose that $n\ge 2$. 
The category $\cC$ in (\ref{def:C}) is given by 
\begin{align*}
\cC=\Coh_{\le n-2}(X)
\end{align*}
and the saturated subgroup $\Gamma_0 \subset \Gamma$
generated by $v^B(\cC)$
is given by
\begin{align*}
\Gamma_0=\{(v_0^B, v_1^B, \cdots, v_n^B) \in \Gamma : 
v_0^B=v_1^B=0 \}.
\end{align*} 
We define the quadratic 
form $\overline{\Delta}_{\omega, B}$ on 
$\Gamma$
by 
\begin{align*}
\overline{\Delta}_{\omega, B}(v) \cneq 
(v_1^B)^2 -2v_0^B v_2^B. 
\end{align*}
\begin{lem}\label{lem:slope}
The very weak 
stability condition 
 $(Z, \Coh(X))$ satisfies the BG inequality
$\overline{\Delta}_{\omega, B}(E) \ge 0$
in the sense of Definition~\ref{defi:BG}. 
\end{lem}
\begin{proof}
For $E \in D^b \Coh(X)$, we have
\begin{align*}
\overline{\Delta}_{\omega, B}(E) = (\omega^{n-1} \ch^B_1(E))^2 - \omega^n \left( \omega^{n-2} \ch_1^B(E)^2 \right) + \omega^n \Delta_\omega(E),
\end{align*}
where $\Delta_\omega(E)$ is defined by
\begin{align*}
\Delta_{\omega}(E)
 = \omega^{n-2} \left(\ch_1^B(E)^2 - 2 \ch_0^B(E) \ch_2^B(E) \right)
\end{align*}
which is independent of $B$. 
The classical Bogomolov-Gieseker inequality 
is $\Delta_{\omega}(E) \ge 0$ for 
any torsion free $\mu_{\omega, B}$-semistable sheaf $E$. 
Together with the Hodge index theorem, the 
inequality $\overline{\Delta}_{\omega, B}(E) \ge 0$
follows
for any $\mu_{\omega, B}$-semistable sheaf
$E \in \Coh(X)$. 
Note that $\overline{\Delta}_{\omega, B}$ can be written as
\begin{align*}
\overline{\Delta}_{\omega, B}(v)=
\Ree Z(v) \cdot (-v_1^B) + \Imm Z(v) \cdot (-2v_2^B). 
\end{align*}
We set 
\begin{align}\label{DRI}
(\Delta_R, \Delta_I)=(-v_1^B, -2v_2^B). 
\end{align} 
Since $v_1^B=0$ on $\Coh_{\le n-2}(X)$, and 
$-2v_2^B(F) \le 0$ for $F \in \Coh_{\le n-2}(X)$
with the equality if and only if 
$F \in \Coh_{\le n-3}(X)$, 
the pair $(Z, \Coh(X))$ satisfies the BG inequality. 
\end{proof}
Note that the category $\cC^{\dag}$ 
defined in (\ref{Cdag0}) 
is given by
\begin{align}\label{Cdag2}
\cC^{\dag}=\Coh_{\le n-3}(X). 
\end{align}
\begin{lem}
The very weak stability condition 
$(Z, \Coh(X))$ is good in the 
sense of Definition~\ref{def:good}. 
\end{lem}
\begin{proof}
The conditions (i), (ii) in Definition~\ref{def:good}
are obvious. 
As for (iii), by (\ref{Cdag2}), the 
saturation of the 
subgroup of $\Gamma_0$
generated by 
$v^B(\cC^{\dag}) \subset \Gamma_0$ 
coincide with $v_2^B=0$ in $\Gamma_0$.  
As for (iv), any torsion free sheaf $U \in \Coh(X)$ fits
 into the short exact sequence
\begin{align*}
0 \to U \to U^{\vee \vee} \to T \to 0
\end{align*}
with $T \in \Coh_{\le n-2}(X)$. 
Since $U^{\vee \vee}$ is reflexive, we have 
$\Hom(\cC, U^{\vee \vee}[1])=0$ as required for condition (iv). 
\end{proof}

\subsection{Tilt stability via very weak stability conditions}
\label{subsec:tilt}
Let $(Z, \Coh(X))$ be 
the very weak stability condition 
given by (\ref{ZCoh}).  
By Corollary~\ref{cor:support}, we  
have the associated very weak stability conditions
\begin{align}\notag
(Z_t^{\dag}, \Coh^{\dag}(X)) \in \Stab_{\omega, B}^{\rm{vw}}(X), 
\ t \in \mathbb{R}_{>0}.
\end{align}
Using (\ref{DRI}), 
the group homomorphism 
$Z_t^{\dag}=-i Z +t \Delta_I$ is given by
\begin{align}\label{Zdag}
Z_t^{\dag}(v)=\left(-2t v_2^B + v_0^B\right) + v_1^B i. 
\end{align}
The associated slope function $\mu_{t}^{\dag}$ is given by
\begin{align}\label{mutd}
\mu_{t}^{\dag}=\frac{2tv_2^B -v_0^B}{v_1^B}. 
\end{align}
Let us describe $\Coh^{\dag}(X)$ in terms of tilting of $\Coh(X)$. 
\begin{defi}\label{def:HNmu}
 For
a given subset  
$I \subset \mathbb{R} \cup \{\infty\}$, the subcategory $\HN^{\mu}_{\omega, B}(I) \subset
\Coh(X)$ is defined by
$$
\HN^{\mu}_{\omega, B}(I) = \langle E \in \Coh(X) :
E \mbox{ is } \mu_{\omega, B} \mbox{-semistable with }
\mu_{\omega, B}(E) \in I \rangle.
$$
Here $\langle \ast \rangle$ means the extension 
closure.
We also define 
the pair of subcategories 
$(\mathcal{T}_{\omega,B}, \mathcal{F}_{\omega,B})$
of $\mathrm{Coh}(X)$ to be
\begin{align*}
\tT_{\omega , B} = \HN^{\mu}_{\omega, B}((0, \infty]), \ \ \
\fF_{\omega , B} = \HN^{\mu}_{\omega, B}((-\infty, 0]).
\end{align*}
\end{defi} 
If $I=\{\vartheta\}$ for some
$\vartheta \in \mathbb{R} \cup \{\infty\}$, we just write 
$\HN^{\mu}_{\omega, B}(I)$ as $\HN^{\mu}_{\omega, B}(\vartheta)$. 
Therefore, in terms of the slicing $\{\pP(\phi)\}_{\phi \in \mathbb{R}}$ for 
the $\mu_{\omega,B}$-stability 
\begin{align*}
\pP(\phi) = \HN^{\mu}_{\omega,B}(-1/\tan (\pi \phi)).
\end{align*}
\begin{rmk}
Since $\ch_1^B = \ch_1 - B \, \ch_0$, we have $\mu_{\omega,B}(E) - \vartheta
 = \mu_{\omega, B+\vartheta\omega}(E)$ for any $\vartheta \in \mathbb{R}$. 
Therefore, for any $r \in \mathbb{R}_{>0}$ we have
\begin{align*}
\HN^{\mu}_{\omega, B}(\vartheta) = \HN^{\mu}_{r\omega, B+\vartheta\omega}(0).
\end{align*}
\end{rmk}
 The existence of Harder-Narasimhan filtrations 
with respect to $\mu_{\omega, B}$-stability 
implies that the pair $(\mathcal{T}_{\omega, B}, \mathcal{F}_{\omega, B})$
is a torsion pair (see~\cite{HRS}) on $\mathrm{Coh}(X)$. 
Let the abelian category 
\begin{align*}
\mathcal{B}_{\omega, B } =
\langle \mathcal{F}_{\omega, B}[1], \mathcal{T}_{\omega,B }
\rangle \subset D^b \mathrm{Coh}(X) 
\end{align*}
be the corresponding tilt of $\Coh(X)$.
Obviously, we have $\Coh^{\dag}(X)=\bB_{\omega, B}$. 
\begin{rmk}
If $\dim X=2$, then $\cC^{\dag}=\{0\}$ and 
$(Z_t^{\dag}, \bB_{\omega, B})$ is a Bridgeland 
stability condition constructed 
by~\cite{Brs2}, \cite{AB}. 
\end{rmk}

From here onwards we assume $\dim X=3$. 
We
relate $(Z_t^{\dag}, \bB_{\omega, B})$
with the tilt stability 
introduced in~\cite{BMT}. 
Let $Z_{\omega,B } \colon K(X) \to \mathbb{C}$
be the group homomorphism 
defined by 
\begin{align}\label{ZwB}
Z_{\omega, B}(E) &=
-\int_{X} e^{-i\omega}\ch^B(E) \\ \notag
&=\left(-v_3^B(E) + \frac{1}{2}v_1^B(E)\right)
+ i\left(v_2^B(E) - \frac{1}{6}v_0^B(E)
  \right).
\end{align}
Note that $Z_{\omega, B}$ factors through 
$v^B \colon K(X) \twoheadrightarrow \Gamma_{\omega, B}$. 
Following \cite{BMT}, 
the tilt-slope $\nu_{\omega , B}(E)$ of $E \in \bB_{\omega , B}$ is defined by
\begin{align*}
\nu_{\omega,B}(E)
=\frac{\Imm Z_{\omega,B }(E)}{v_1^B(E)}.
\end{align*}
Here we set $\nu_{\omega, B}(E)=\infty$ when $v_1^B(E) = 0$. 
In the above notation, 
we have 
\begin{align*}
6\nu_{\omega, B}(E)=\mu_{t=3}^{\dag}(E)
\end{align*} 
from the formula (\ref{mutd}). 
The associated 
$\nu_{\omega, B}$-stability on $\bB_{\omega, B}$
was called \textit{tilt-stability} in~\cite{BMT}. 
\begin{rmk}\label{rmk:tiltstab}
By~\cite[Corollary~3.3.3]{BMT}, the tilt 
stability can be 
extended to any $B+i\omega \in \mathrm{NS}(X)_{\mathbb{C}}$
with $\omega$ ample by~\cite[Corollary~3.3.3]{BMT}. 
So 
we don't have to assume the condition (\ref{class:choice})
to discuss the tilt stability.
We keep the condition (\ref{class:choice}) until the end of 
this subsection, while we 
discuss without the condition (\ref{class:choice}) in the next subsection. 
 \end{rmk}
Let us discuss the BG inequality for 
$(Z_{t}^{\dag}, \bB_{\omega, B})$. 
By (\ref{Cdag2}), 
the category $\cC^{\dag}$ is given by 
$\Coh_0(X)$. 
Therefore, a natural choice of the 
correction term $\nabla_I$ in (\ref{Del:dag2}) may be of the 
form $a v_3^B$ for some $a<0$. 
By (\ref{DRI}), 
the functions (\ref{DRIdag})
may be of the form
\begin{align*}
(\Delta_R^{\dag}, \Delta_I^{\dag})=(-2v_2^B, v_1^B +av_3^B)
\end{align*}
and 
the inequality (\ref{Del:dag2}) may be of the following form
\begin{align*}
(-2tv_2^B+v_0^B) \cdot (-2v_2^B)+ v_1^B \cdot (v_1^B + av_3^B) \ge 0. 
\end{align*}
We require that 
the above inequality 
becomes an equality when $v_i^B=1/i^{!}$; i.e. 
when $B$ is proportional to $\omega$, it is 
the Chern character of a line bundle $L$
with $c_1(L)$ proportional to $\omega$.
This requirement uniquely determines 
$a=-6t$, and the above inequality becomes
\begin{align*}
(v_1^B)^2 -2v_0^B v_2^B +2t(2(v_2^B)^2 -3v_1^B v_3^B) \ge 0. 
\end{align*}
Now we define the quadratic form $\overline{\nabla}_{\omega, B}$
on $\Gamma$ to be
\begin{align*}
\overline{\nabla}_{\omega, B}(v)
\cneq 2 (v_2^B)^2 -3 v_1^B v_3^B. 
\end{align*}
By setting $t=3$, we arrive at the following conjecture: 
\begin{conj}\label{conj:BMS}
Let $X$ be a smooth projective 3-fold. 
Then for any $\nu_{\omega, B}$-semistable object 
$E \in \mathcal{B}_{\omega, B}$, 
we have the inequality
\begin{align}\label{ineq:BMS}
\overline{\Delta}_{\omega, B}(E) + 6 \overline{\nabla}_{\omega, B}(E) \ge 0.  
\end{align}
\end{conj}
Note that by setting $\nu_{\omega, B}(E)=0$, we obtain the following
conjecture stated in~\cite{BMT}: 
\begin{conj}[\cite{BMT}]\label{conj:BMT}
Let $X$ be a smooth projective 3-fold. 
Then for any $\nu_{\omega, B}$-semistable object 
$E \in \mathcal{B}_{\omega, B}$ with 
$\nu_{\omega, B}(E)=0$, i.e. $v_1^B(E) \neq 0$,
$v_2^B(E)=v_0^B(E)/6$, 
we have the inequality
\begin{align*}
v_3^B(E) \le \frac{1}{18} v_1^B(E). 
\end{align*}
\end{conj}
Obviously Conjecture~\ref{conj:BMT} is implied by Conjecture~\ref{conj:BMS}. 
Conversely from \cite[Theorem~4.2]{BMS}, 
if Conjecture~\ref{conj:BMT} holds for all $(B, \omega)$
with $B$ proportional to $\omega$, then Conjecture~\ref{conj:BMS}
also holds for all $(B, \omega)$ with $B$ proportional to $\omega$. 
In Subsection~\ref{subsec:Generalized}, 
we will generalize the result of~\cite[Theorem~4.2]{BMS}, 
and show that Conjecture~\ref{conj:BMS} and Conjecture~\ref{conj:BMT}
are indeed equivalent. 

\subsection{Double tilting construction}\label{subsec:double}
In this subsection, we change the notation 
and set $W=Z_{t=3}^{\dag}$
where $Z_t^{\dag}$ is given by (\ref{Zdag}), i.e. 
\begin{align}\label{Ztilt}
W(v)=\left(-6v_2^B + v_0^B\right) + v_1^B i. 
\end{align}
We consider the very weak stability condition $(W, \bB_{\omega, B})$. 
Note that the associated semistable objects 
coincide with $\nu_{\omega, B}$-semistable objects
in $\bB_{\omega, B}$.  

\begin{prop}\label{prop:good}
Suppose that Conjecture~\ref{conj:BMS} holds. 
Then $(W, \bB_{\omega, B})$
satisfies the BG inequality and
is good.  
\end{prop}
\begin{proof}
The inequality in 
Conjecture~\ref{conj:BMS} implies 
that 
for any $\nu_{\omega, B}$-semistable object $E \in \bB_{\omega, B}$, 
we have 
\begin{align*}
\Ree W(E) \Delta_R^{\dag}(E)+\Imm W(E) \Delta_I^{\dag}(E) \ge 0,
\end{align*}
where
\begin{align}\label{Delta3}
(\Delta_R^{\dag}, \Delta_I^{\dag})=(-2v_2^B, v_1^B -18v_3^B).
\end{align}
Note that the category $\cC^{\dag} \subset \bB_{\omega, B}$
of objects $E \in \bB_{\omega, B}$ with $W(E)=0$
coincides with $\Coh_0(X)$. 
Since $\Delta_R|_{\Coh_0(X)} \equiv 0$
and 
$\Delta_I|_{\Coh_0(X) \setminus \{0\}} <0$, 
the very weak stability
condition $(W, \bB_{\omega, B})$ satisfies the BG
inequality. 

We check that $(W, \bB_{\omega, B})$ is good. 
The condition (i) in Definition~\ref{def:good} is obvious.
The condition (ii) follows from 
Lemma~\ref{lem:noether} since  
$\bB_{\omega, B}$ is a tilt of a good 
very weak stability condition on $\Coh(X)$.   
Also   
the category $\cC^{\dag \dag}$ of objects $F \in \cC^{\dag}=\Coh_0(X)$
with $\Delta_I^{\dag}(F)=0$ is $\{0\}$.  
Therefore, we also have the conditions (iii) and (iv) as required. 
\end{proof}
As we observed in
 the proof of Proposition~\ref{prop:good}, 
the category $\cC^{\dag \dag}$ is zero.
Hence, by Corollary~\ref{cor:support} 
we have the associated one parameter family of stability conditions
\begin{align}\label{dtilt}
(W_t^{\dag}, \bB_{\omega, B}^{\dag}) \in \Stab_{\omega, B}(X), \ 
t \in \mathbb{R}_{>0}. 
\end{align}
The group homomorphism $W_t^{\dag}=-i W+t \Delta_I^{\dag}$
 is given by
\begin{align}\notag
W_t^{\dag}(v)=\left((1+t)v_1^B-18tv_3^B\right) +(6v_2^B-v_0^B)i. 
\end{align}
The associated slope function $\nu_t^{\dag}$ is given by 
\begin{align}\notag
\nu_t^{\dag}=
\frac{18tv_3^B-(1+t)v_1^B}{6v_2^B-v_0^B}. 
\end{align}
By comparing it with $Z_{\omega, B}$
given by (\ref{ZwB}), 
we have 
\begin{align}\label{nu=Z}
\frac{9}{4}\nu_{t=1/8}^{\dag}=-\frac{\Ree Z_{\omega, B}}{\Imm Z_{\omega, B}}. 
\end{align}
We describe $\bB_{\omega, B}^{\dag}$ in terms of 
tilt stability. 
Similar to Definition~\ref{def:HNmu}, for tilt-slope $\nu_{\omega,B}$ stability
on $\bB_{\omega, B}$ we define the following: 
\begin{defi}
 For
a given subset  
$I \subset \mathbb{R} \cup \{\infty\}$, the subcategory $\HN^{\nu}_{\omega, B}(I) \subset
\bB_{\omega, B}$ is defined by
$$
\HN^{\nu}_{\omega, B}(I) = \langle E \in \bB_{\omega, B} :
E \mbox{ is } \nu_{\omega, B} \mbox{-semistable with }
\nu_{\omega, B}(E) \in I \rangle.
$$
We also define 
the pair of subcategories 
$(\mathcal{T}_{\omega,B}', \mathcal{F}_{\omega,B}')$
of $\bB_{\omega, B}$ to be
\begin{align*}
\tT_{\omega , B}' = \HN^{\nu}_{\omega, B}((0, \infty]), \ \ \
\fF_{\omega , B}'= \HN^{\nu}_{\omega, B}((-\infty, 0]).
\end{align*}
\end{defi}
We also set 
$\HN^{\nu}_{\omega, B}(I)=\HN^{\nu}_{\omega, B}(\vartheta)$
when $I=\{\vartheta\}$. 
The HN 
property of tilt stability implies that
the pair of subcategories 
$(\mathcal{T}_{\omega,B}', \mathcal{F}_{\omega,B}')$
forms a torsion pair on $\bB_{\omega,B}$. 
By tilting, we have another heart
\begin{align*}
\mathcal{A}_{\omega, B} = \langle \mathcal{F}_{\omega, B}'[1], 
\mathcal{T}_{\omega, B}' \rangle \subset D^b \mathrm{Coh}(X). 
\end{align*}
By the construction, 
we have $\bB_{\omega, B}^{\dag}=\aA_{\omega, B}$. 
By (\ref{nu=Z}), the condition
(\ref{dtilt}) for $t=1/8$ is equivalent 
to the statement
\begin{align}\label{BMT:stab}
\sigma_{\omega, B} \cneq 
(Z_{\omega, B}, \mathcal{A}_{\omega, B}) \in 
\mathrm{Stab}_{\omega, B}(X),
\end{align}
which was conjectured in~\cite{BMT}. 




\subsection{Slope bounds for cohomology sheaves of objects in $\bB_{\omega,B}$}
In this  and in the next subsections, 
we prepare some results on tilt stability
to show the equivalence of Conjectures~\ref{conj:BMS}
and \ref{conj:BMT}. 
We will carry the notations introduced in Subsections \ref{subsec:tilt} and \ref{subsec:double}, except that we now don't assume the condition 
(\ref{class:choice})
as we mentioned in Remark~\ref{rmk:tiltstab}. 

It is straightforward to check that, for any $\vt \in \mathbb{R}$
\begin{equation}
\label{eqn:chern}
v_i^{B+\vt \omega} = \sum_{j=0}^{i} \frac{(-1)^j}{j!} \vt^j \, v_{i-j}^B,
\end{equation}
and for any $\beta \in \mathbb{R}_{>0}$, we have 
$\overline{\Delta}_{\vt \omega,B+ \beta \omega}(E) = \vt^4 \overline{\Delta}_{\omega,B}(E)$.
For a given  $\vt \in \mathbb{R}$, we set 
$$
\eta = \sqrt{3\vt^2+1}.
$$
\begin{lem}
For $E \in D^b \Coh(X)$, we have the following identities:
\begin{align} \label{Duid1}
&\Imm Z_{\omega,B}(E) - \vt v_1^B(E) = \frac{1}{\eta} \Imm Z_{\eta \omega, B+ \vt\omega}(E) \\
\label{Duid2}
&v_0(E) \Imm Z_{\omega,B}(E) = -\frac{1}{2} \overline{\Delta}_{\omega,B}(E) + \frac{1}{2} v_1^{B-\omega/\sqrt{3}}(E) v_1^{B+\omega/\sqrt{3}}(E).
\end{align}
\end{lem}
\begin{proof}
The identity (\ref{Duid1}) follows easily by simplifying $\Imm Z_{\eta \omega, B+ \vt\omega}(E)$ with the relations \eqref{eqn:chern}.

Let us show the identity (\ref{Duid2}). 
Let $B' = B - \omega/\sqrt{3}$. 
By the definition of $\Imm Z_{\omega,B}(E)$, and  
$\ch^B(E) = e^{-\omega/\sqrt{3}}\ch^{B'}(E)$, we have 
$$
\Imm Z_{\omega,B}(E) = v_2^{B'}(E) - \frac{1}{\sqrt{3}}v_1^{B'}(E).
$$
Therefore,
\begin{align*}
 &v_0(E)  \Imm Z_{\omega,B}(E) \\
&\qquad = v_0(E) v_2^{B'}(E) - \frac{1}{2} (v_1^{B'}(E))^2  + \frac{1}{2} v_1^{B'}(E) \left(v_1^{B'}(E) - \frac{2}{\sqrt{3}}v_0(E)\right) \\
 & \qquad = -\frac{1}{2} \overline{\Delta}_{\omega,B}(E) + \frac{1}{2} v_1^{B'}(E) v_1^{B'+2\omega/\sqrt{3}}(E)   \\
 &\qquad = -\frac{1}{2} \overline{\Delta}_{\omega,B}(E) + \frac{1}{2} v_1^{B-\omega/\sqrt{3}}(E) v_1^{B+\omega/\sqrt{3}}(E)
\end{align*}
as required. 
\end{proof}
\noindent Consequently, we have 
\begin{align*}
\nu_{\omega, B} (E) - \vt &= \frac{\Imm Z_{\omega,B}(E) - \vt v_1^B(E)}{v_1^B(E)} \\ 
&=\frac{ \Imm Z_{\eta \omega, B+ \vt\omega}(E)}{\eta v_1^B(E)}.
\end{align*}
Assume $v_0^B(E)  \ne 0$. From multiplying both the denominator and the 
numerator by $v_0^B(E)$, we get
\begin{align}
 \nu_{\omega, B} (E) - \vt 
 =  \frac{-\overline{\Delta}_{\omega, B}(E) 
+ v_1^{B + (\vt-\eta/\sqrt{3})\omega}(E)  
v_1^{B+ (\vt+ \eta/\sqrt{3})\omega}(E)}
 {2\eta v_0^B(E) v_1^B(E)}.
 \label{eq1}
\end{align}

We have the following proposition, which 
generalizes the claims in \cite[Prop.~3.1, 3.2]{MaPi1} and 
\cite[Prop.~4.1, 4.2]{MaPi2}. Our proofs are also somewhat similar to them. 
\begin{prop}
\label{prop:slope-bounds}
Let $\vt$ be any real number and let $\eta = \sqrt{3\vt^2+1}$. 
Let $E \in \bB_{\omega, B}$ and $E_{i}= \hH^{i}(E)$.
Then we have the following:
 \begin{enumerate}[(i)]
\item if $E \in \HN^{\nu}_{\omega, B}((-\infty, \vt)) $,  then 
$E_{-1} \in \HN^{\mu}_{\omega, B} ((-\infty, \vt- \eta/\sqrt{3}))$;

\item if $E \in \HN^{\nu}_{\omega, B}((\vt, \infty)) $,  then 
$E_0 \in \HN^{\mu}_{\omega, B} (( \vt + \eta/\sqrt{3}, \infty])$;
and

\item if $E$ is tilt semistable with $\nu_{\omega,B}(E) =\vt$, then
\begin{enumerate}[(a)]
\item $E_{-1} \in \HN^{\mu}_{\omega, B} ((-\infty, \vt- \eta/\sqrt{3}])$ with equality
$\mu_{\omega,B}(E_{-1}) = \vt - \eta/\sqrt{3}$ holds
  if and only if $v_2^{B+(\vt- \eta/\sqrt{3})\omega}(E_{-1}) = 0$, and
\item when $E_0$ is torsion free $E_0 \in \HN^{\mu}_{\omega, B} ([\vt + \eta/\sqrt{3}, \infty))$ with equality
$\mu_{\omega,B}(E_{0}) = \vt + \eta/\sqrt{3}$ holds
  if and only if $v_2^{B+(\vt+\eta/\sqrt{3})\omega}(E_{0}) =0$.
\end{enumerate}
\end{enumerate}
\end{prop}
\begin{proof}
Note that any object $E \in \bB_{\omega,B}$ fits into the short exact sequence
$$
0 \to E_{-1}[1] \to E \to E_0 \to 0
$$
in  $\bB_{\omega,B}$.

(i) If $E \in \HN^{\nu}_{\omega, B}((-\infty, \vt))$ then we have $E_{-1}[1]  \in \HN^{\nu}_{\omega, B}((-\infty, \vt))$.
Let $E_{-1}^+$ be the HN $\mu_{\omega, B}$-semistable factor of 
$E_{-1}$ with the highest $\mu_{\omega, B}$ slope.
From the HN filtration of $E_{-1}$ with respect to the 
$\mu_{\omega, B}$-stability, $E_{-1}[1]$ fits into the short exact sequence
$$
 0 \to  E_{-1}^+ [1] \to E_{-1}[1] \to Q [1] \to 0
$$
in $\bB_{\omega,B}$, where 
$Q$ is the cokernel of $E_{-1}^{+} \hookrightarrow E_{-1}$ in 
$\Coh(X)$. 
Hence, $E_{-1}^+[1] \in \HN^{\nu}_{\omega,B}((-\infty, \vt))$.

Assume the opposite for a contradiction; so that
$$
v_1^{B+(\vt- \eta/\sqrt{3})\omega}(E_{-1}^+) \ge 0.
$$
We have
$$
  \nu_{\omega,B} (E_{-1}^+[1]) -\vt  =  \nu_{\omega,B} (E_{-1}^+) -\vt.
$$
Since $E_{-1}^+$ is $\mu_{\omega, B}$-semistable, by the classical Bogomolov-Gieseker inequality
$$
\overline{\Delta}_{\eta\omega,B}(E_{-1}^+) \ge 0
$$
and since $E_{-1}^+ \in \fF_{\omega,B} = \HN^{\mu}_{\omega,B}((-\infty, 0])$, 
we have
 $\nu_{\omega,B}(E_{-1}^+[1]) \ne \infty$. Also for any $\vt \in \mathbb{R}$
we have $\vt + \eta/\sqrt{3}> 0$, and so
$$
 v_1^B(E_{-1}^+) < 0 \  \text{and} \  v_1^{B + (\vt + \eta/\sqrt{3})\omega}(E_{-1}^+)<0.
$$ 
Hence, as $v_0(E_{-1}^+) >0$, by \eqref{eq1}, we have $\nu_{\omega,B}
(E_{-1}^+) - \vt \ge 0$.
But this is not possible as $E_{-1}^+[1] \in
\HN^{\nu}_{\omega,B}((-\infty, \vt))$. This is the required
contradiction to complete the proof.

(ii) 
If $v_0(E_0)=0$ then clearly we have the required result. So we may assume
 $v_0(E_0) > 0$.
Since $E \in \HN^{\nu}_{\omega, B}((\vt, \infty])$, we have
$E_0 \in \HN^{\nu}_{\omega,B}((\vt, \infty])$.
Let $E_{0}^-$ be the HN $\mu_{\omega, B}$-semistable factor of $E_{0}$ with the lowest $\mu_{\omega,B}$ slope.
By the HN filtration of $E_{0}$ with respect to 
$\mu_{\omega, B}$-stability, 
since $v_0(E_0) > 0$ we have
$ E_{0}^-$ is a non-trivial torsion free quotient of $E_0$.
So
$$
0 \to K \to E_0 \to  E_{0}^- \to 0
$$
is a short exact sequence in $\bB_{\omega,B}$, where 
$K$ is the kernel of 
$E_0 \twoheadrightarrow E_0^{-}$ in 
$\Coh(X)$. 
Since
$E_0 \in \HN^{\nu}_{\omega,B}((\vt, \infty])$ we have $E_0^- \in \HN^{\nu}_{\omega,B}((\vt, \infty])$ .

Assume the opposite for a contradiction; 
so that 
$$
v_1^{B+(\vt+\eta/\sqrt{3})\omega}(E_{0}^-) \le 0.
$$
Since $E_{0}^-$ is $\mu_{\omega, B}$-semistable, by the classical Bogomolov-Gieseker inequality,
$$
\overline{\Delta}_{\eta \omega,B}(E_{0}^-) \ge 0. 
$$
Also since $E_{0}^- \in \tT_{\omega,B} = \HN^{\mu}_{\omega,B}((0,\infty])$ is torsion free, 
for any $\vt \in \mathbb{R}$ we have $\vt - \eta/\sqrt{3}< 0$, and so

$$
v_1^B(E_{0}^-) > 0 \ \text{and} \ v_1^{B + (\vt- \eta/\sqrt{3})\omega}(E_{0}^-)>0.
$$ 
Hence, as $v_0(E_{0}^-) >0$, by \eqref{eq1}, we have $\nu_{\omega,B}
(E_{0}^-) - \vt \le 0$.
But this is not possible as $E_{0}^- \in
\HN^{\nu}_{\omega,B}((\vt,\infty])$. This is the required
contradiction to complete the proof.

(iii)
Similarly to (i) one can show that if $E \in \HN^{\nu}_{\omega,B}((-\infty, \vt])$ then
$E_{-1} \in \HN^{\mu}_{\omega, B} ((-\infty, \vt- \eta/\sqrt{3}])$.
Hence, for  tilt semistable $E \in \bB_{\omega,B}$ with $\nu_{\omega,B}(E)=\vt$, we have 
 $E_{-1} \in \HN^{\mu}_{\omega, B} ((-\infty, \vt- \eta/\sqrt{3}])$.
The equality 
$v_1^{B+(\vt- \eta/\sqrt{3})\omega}(E_{-1}^+) = 0$
holds when $E_{-1}$ is slope semistable, and so it satisfies the classical
BG inequality.
Since $\nu_{\omega,B}(E_{-1})-\vt \le 0$, from \eqref{eq1} we have 
$v_1^{B+(\vt- \eta/\sqrt{3})\omega}(E_{-1})=0$ if and only if $\overline{\Delta}_{\omega,B}(E_{-1}) = 0$ equivalently $v_2^{B+(\vt- \eta/\sqrt{3})\omega}(E_{-1})=0$.
This completes the proof of (a).
The proof of (b) is similar to that of (a).
\end{proof}


\subsection{Some properties of tilt stable objects}
When $E \in \Coh(X)$ is a $\mu_{\omega,B}$-(semi)stable sheaf, it is straightforward to check that 
it is also $\mu_{r \omega, B+\vt\omega}$-(semi)stable for any $r \in \mathbb{R}_{>0}$ and $\vt \in \mathbb{R}$.
In this subsection we obtain a somewhat similar 
result for tilt stable objects.

\begin{lem}
\label{prop2.1}
Let $E \in \bB_{\omega,B}$ and for given $\vt \in \mathbb{R}$ let $\eta = \sqrt{3\vt^2+1}$. 
Then $E$ is $\nu_{\omega,B}$-stable with $\nu_{\omega,B} (E)=\vt$ if and only if 
$E \in \bB_{\eta\omega, B+\vt\omega}$ is $\nu_{\eta\omega, B+\vt\omega}$-stable with $\nu_{\eta\omega, B+\vt\omega}(E) =0$.
\end{lem}
\begin{proof}
Suppose that $E \in \bB_{\omega, B}$ is $\nu_{\omega, B}$-stable 
with $\nu_{\omega, B}(E)=\vartheta$. 

Let $E_{i} := \hH^{i}(E)$.
By (iii) of Proposition
\ref{prop:slope-bounds}, we have 
$$
E_{-1} \in \HN_{\omega, B}^{\mu}((-\infty, \vt-\eta/\sqrt{3}]) \subset \HN_{\eta\omega, B+\vt\omega}^{\mu}((-\infty, 0)),
$$
and
$$
E_0 \in \HN_{\omega, B}^{\mu}([\vt+\eta/\sqrt{3},\infty]) \subset \HN_{\eta\omega, B+\vt\omega}^{\mu}((0,\infty]).
$$
Therefore, $E \in \bB_{\eta\omega, B+\vt\omega}$. 
Recall that 
$$
\nu_{\omega,B} (E) -\vt = \frac{ \Imm Z_{\eta \omega, B+ \vt\omega}(E)}{\eta v_1^B(E)}.
$$
Therefore, $\Imm Z_{\eta \omega, B+ \vt\omega}(E) = 0$ and so $\nu_{\eta \omega, B+ \vt\omega}(E)=0$. 

Assume $E$ is a $\nu_{\eta \omega, B+ \vt\omega}$ tilt unstable object for a contradiction. 
Therefore, from the Harder-Narasimhan property,  there exist 
$F \in \HN_{\eta \omega, B+ \vt\omega}^{\nu}((0,\infty])$, 
$G \in \HN_{\eta \omega, B+ \vt\omega}^{\nu}((-\infty,0))$, and
a short exact sequence
\begin{equation}
\label{ses.1}
0 \to F \to E \to G \to 0
\end{equation}
 in $\bB_{\eta \omega, B+ \vt\omega}$. By considering the long exact sequence of $\Coh(X)$-cohomologies, we have $F_{-1} \hookrightarrow E_{-1}$ and $E_0 \twoheadrightarrow G_0$ in $\Coh(X)$.  So 
$$
F_{-1} \in \HN_{\omega, B}^{\mu}((-\infty, \vt-\eta/\sqrt{3}])  \subset \HN_{\omega, B}^{\mu}((-\infty, 0)),
$$
and 
 $$
G_0 \in \HN_{\omega, B}^{\mu}([\vt+\eta/\sqrt{3},\infty])  \subset \HN_{\omega, B}^{\mu}((0,\infty]).
 $$
By (ii) of Proposition \ref{prop:slope-bounds}, 
$$
F_0  \in \HN_{\eta\omega, B+ (\vt+\eta/\sqrt{3})\omega}^{\mu}((0,\infty])  \subset \HN_{\omega, B}^{\mu}((0,\infty]),
$$
and by (i) of Proposition \ref{prop:slope-bounds}, 
 $$
 G_{-1} \in  \HN_{\eta\omega, B + (\vt-\eta/\sqrt{3})\omega}^{\mu}((-\infty, 0))  \subset \HN_{\omega, B}^{\mu}((-\infty, 0)).
 $$
Therefore, \eqref{ses.1} is also a short exact sequence in $\bB_{\omega,B}$.
Since we have  $\Imm Z_{\eta \omega, B+ \vt\omega}(F)>0$ and $\Imm Z_{\eta \omega, B+ \vt\omega}(G) <0$, 
we obtain 
$\nu_{\omega,B}(F)- \vt> 0$ and $\nu_{\omega,B}(G) -\vt< 0$. 
That is $E$ is $\nu_{\omega,B}$-unstable; which is not possible. This is the required contradiction to complete the proof. 
The proof in the other direction is similar to the above. 
\end{proof}

\noindent Let $\alpha \in \mathbb{R}_{>0}$ and $\beta \in \mathbb{R}$.
Then from the above result we have 
$$
\HN^{\nu}_{\alpha \omega, B + \beta \omega}(\tau) = \HN^{\nu}_{\alpha\sqrt{3\tau^2+1}  \omega, B+ (\beta+ \tau \alpha)\omega}(0).
$$
Moreover, we have 
$$
\HN^{\nu}_{\omega, B}(\vt) = \HN^{\nu}_{\sqrt{3\vt^2+1}\omega, B+ \vt\omega}(0).
$$
Therefore,  
$
\HN^{\nu}_{\omega, B}(\vt) = \HN^{\nu}_{\alpha \omega, B + \beta \omega}(\tau) 
$
when $\sqrt{3\vt^2+1} = \alpha\sqrt{3\tau^2+1} $ and 
$\vt = \beta + \tau \alpha$ hold.
So we get the following, which is a generalization of \cite[Lemma~4.3]{BMS}.
\begin{lem}
\label{lem:wall-tiltstable}
Let the object $E \in \bB_{\omega, B}$ be $\nu_{\omega,B}$-stable. Then $E \in \bB_{\alpha \omega, B + \beta \omega}$ is 
$\nu_{\alpha \omega, B + \beta \omega}$-stable for all $\alpha \in \mathbb{R}_{>0}$ and $\beta \in \mathbb{R}$ such that
$$
\alpha^2  + 3 \left( \beta - \nu_{\omega,B}(E)\right)^2  = 3\nu_{\omega,B}(E)^2 + 1.
$$
\end{lem}

\subsection{Generalized conjectural BG inequality}\label{subsec:Generalized}

Let $E \in \bB_{\omega,B}$ be a $\nu_{\omega,B}$-stable object with $\nu_{\omega,B}(E) =\vt$ for some $\vt \in \mathbb{R}$. 
So 
\begin{equation}
\label{eq2}
v_2^B(E) - \frac{1}{6}v_0^B(E)= \vt v_1^B(E).
\end{equation}

From Proposition \ref{prop2.1}, $E \in \bB_{\eta\omega, B+\vt\omega}$ is $\nu_{\eta\omega, B+\vt\omega}$-stable with $\nu_{\eta\omega, B+\vt\omega}(E) =0$. 
So the inequality for $E$ in Conjecture~\ref{conj:BMT}
reads as:
$$
v_3^{B+\vt\omega}(E) \le \frac{(3\vt^2+1)}{18}v_1^{B+\vt\omega}(E).
$$
That is,
$$
v_3^B(E) - \vt v_2^B(E) + \frac{(6\vt^2-1)}{18} v_1^B(E) + \frac{\vt}{18} v_0^B(E) \le 0.
$$
 As $\vt \ne \infty$, $v_1^B(E) > 0$. Therefore, from multiplying both sides of the above inequality by 
 $v_1^B(E)$ and then using \eqref{eq2}, we get the inequality
\begin{align*}
18v_1^B(E) v_3^B(E) - 12(v_2^B(E))^2 + 2  v_0^B(E) v_2^B(E) - (v_1^B(E))^2 \le 0.
\end{align*}
Since $\overline{\Delta}_{\omega,B}(E) = (v_1^B(E))^2 -  2  v_0^B(E) v_2^B(E)$ 
and $\overline{\nabla}_{\omega, B}(E) = 2(v_2^B(E))^2 - 3v_1^B(E)v_3^B(E)$,
 we obtain the following result: 
\begin{thm}\label{thm:BG:equiv}
For a given smooth projective 3-fold $X$, 
Conjecture \ref{conj:BMS} holds if and only if 
Conjecture \ref{conj:BMT}  holds for all the
complexified ample classes $B + i \omega$.
\end{thm}

\subsection{Another equivalent form of BG inequality conjecture}

In this subsection we formulate an equivalent form of Conjecture \ref{conj:BMS} which only considers BG type inequalities for  a small class of tilt stable objects. 
This generalizes \cite[Conjecture 5.3]{BMS} and we show that it is equivalent to Conjecture \ref{conj:BMS}. 
The following discussion is not needed in the rest of the paper, and so the reader is safe to skip this subsection.

Most of our arguments in this subsection are closely related to Section 5 of \cite{BMS}, and  also we try to follow somewhat similar notations. 

Let us consider the complexified classes parametrized by $\alpha \in \mathbb{R}_{>0}$ and $\beta \in \mathbb{R}$:
$$
(B + \beta \omega) + i (\alpha \omega).
$$
By definition  
$$
\nu_{\alpha \omega, B+ \beta\omega} = \frac{\omega \ch_2^{B+\beta\omega} - \frac{1}{6}\alpha^2\omega^3 \ch_0}{\alpha \omega^2 \ch_1^{B+\beta\omega}}.
$$
Therefore, we can consider 
$$
Z_{\alpha, \beta}^{\nu} = -\left( \omega \ch_2^{B+\beta\omega} - \frac{1}{6}\alpha^2\omega^3 \ch_0 \right) + i \omega^2 \ch_1^{B+\beta\omega}
$$
as the associated group homomorphism in $\nu_{\alpha \omega, B+ \beta\omega}$ tilt stability.
For a given object $E$, if we have
$$
\lim_{\alpha \to 0^+} - \Ree Z_{\alpha , \beta}^\nu(E) =0,
$$ 
when  $\beta \to \overline{\beta}$, then $\overline{\beta}$  satisfies $\omega \ch_2^{B+\overline{\beta} \omega}(E) = 0$.
That is
$$
v_0^B(E) \, \overline{\beta}^2 - 2v_1^B(E) \, \overline{\beta} + 
2 v_2^B(E) = 0.
$$
We consider one of its root defined by
\begin{equation}
\label{eq:limit-beta}
\overline{\beta}(E) \cneq \frac{2 v_2^B(E)}{v_1^B(E) + \sqrt{\overline{\Delta}_{\omega, B}(E)}}.
\end{equation}
We conjecture the following, which generalizes  \cite[Conjecture 5.3]{BMS}. Also our claim is directly adapted from their formulation.
\begin{conj}
\label{conj:limitBG}
Let $E$ be an object of $D^b \Coh(X)$.  Suppose there is an open neighborhood $U \subset \mathbb{R}^2$ containing $(0, \overline{\beta}(E))$ such that for any $(\alpha, \beta) \in U$ with $\alpha>0$, 
$E \in \bB_{\alpha\omega, B+\beta\omega}$ is $\nu_{\alpha\omega, B+\beta\omega}$-stable. Then 
$$
\ch_3^{B + \overline{\beta}(E) \omega}(E) \le 0.
$$ 
\end{conj}

Let $E \in \bB_{\omega, B}$ be $\nu_{\omega, B}$-stable. Recall  Lemma \ref{lem:wall-tiltstable}: $E \in \bB_{\alpha \omega, B+ \beta \omega}$ is $\nu_{\alpha \omega, B+\beta\omega}$-stable along the wall in $\mathbb{R}^2_{\alpha, \beta}$ defined by 
\begin{equation*}
\label{eq:wall-tilt}
\wW_{\omega, B}(E) : = \left\{ (\alpha, \beta):  \alpha^2  + 3 \left( \beta - \nu_{\omega,B}(E)\right)^2  = 3\nu_{\omega,B}(E)^2 + 1, \ \alpha >0 \right\}.
\end{equation*}
Let $\dD_{\omega, B}(E)$ be the interior of $\wW_{\omega, B}(E)$ in $\alpha \ge 0$. That is for any $E \in \bB_{\omega, B}$, we define 
\begin{equation*}
\label{eq:interior-wall}
\dD_{\omega, B}(E) : =  \left\{  (\alpha, \beta):  \alpha^2  + 3 \left( \beta - \nu_{\omega,B}(E)\right)^2  <  3\nu_{\omega,B}(E)^2 + 1, \ \alpha \ge 0 \right\}.
\end{equation*}

The $\overline{\Delta}_{\omega, B}$  values  of objects in $\bB_{\omega,B}$ are very important for us. 
In particular,   \cite[Corollary 7.3.2]{BMT} says if  $E \in \bB_{\omega,B}$ is $\nu_{\omega,B}$-stable, 
then 
\begin{equation*}
\overline{\Delta}_{\omega, B}(E) \ge 0.
\end{equation*}
Moreover, we have the following:
\begin{prop}
\label{prop:BG-zero-delta}
The inequalities in Conjectures \ref{conj:BMS} and \ref{conj:limitBG} hold for corresponding tilt stable objects $E$ with 
$\overline{\Delta}_{\omega, B}(E) = 0$. 
\end{prop}
\begin{proof}
Similar to the proof of \cite[Lemma 5.6]{BMS}. 
\end{proof}
The following result generalizes the claim in \cite[Lemma 5.5]{BMS}.
\begin{prop}
\label{prop:beta-limit-position}
Let $E \in \bB_{\omega, B}$ be $\nu_{\omega, B}$-stable with $\overline{\Delta}_{\omega, B}(E)>0$. Then 
$$
(0, \overline{\beta}(E)) \in \dD_{\omega, B}(E).
$$
\end{prop}
\begin{proof}
Let $E \in \bB_{\omega, B}$ be $\nu_{\omega, B}$-stable with $\nu_{\omega, B}(E) = \vt$ for some $\vt \in \mathbb{R}$. 
So $v_2^B(E) = \frac{1}{6} v_0^B(E)  + \vt v_1^B(E)$, and 
by \eqref{eq:limit-beta}
\begin{equation}
\label{eq:beta-lim-tilt}
\overline{\beta}(E) =  \frac{\frac{v_0^B(E)}{3 v_1^B(E)}   + 2 \vt}{1 + \frac{\sqrt{ \overline{\Delta}_{\omega, B}(E) }}{v_1^B(E)} }.
\end{equation}
By Proposition \ref{prop:slope-bounds}, we have 
\begin{align*}
&\omega^2 \ch_1^{B + (\vt + \sqrt{\vt^2 + 1/3})\omega} \left(\hH^{0}(E)\right) \ge 0, \\
&\omega^2 \ch_1^{B + (\vt - \sqrt{\vt^2 + 1/3})\omega} \left(\hH^{-1}(E)\right) \le 0.
\end{align*}
Since 
$\vt + \sqrt{\vt^2 + 1/3} > 0$ and  $ \vt - \sqrt{\vt^2 + 1/3} < 0$, we have
\begin{align*}
v_1^B(E) - \left( \vt + \sqrt{\vt^2 + {1}/{3}}\right) v_0^B(E) \ge & 0,\\
v_1^B(E) - \left( \vt - \sqrt{\vt^2 + {1}/{3}}\right) v_0^B(E) \ge & 0.
\end{align*}
Therefore, by dividing $v_1^B(E)>0$,
\begin{equation*}
 - \vt  - \sqrt{\vt^2 +{1}/{3}} \,  \le \, \frac{v_0^B(E)}{3 v_1^B(E)} \, \le \, - \vt + \sqrt{\vt^2 +{1}/{3}} .
\end{equation*}
Since $v_1^B(E) > 0$ and $ \overline{\Delta}_{\omega, B}(E) > 0$, 
from \eqref{eq:beta-lim-tilt} together with the above inequalities  we get
\begin{equation*}
 \vt - \sqrt{\vt^2 + {1}/{3}} \, < \, \overline{\beta}(E) \, < \,  \vt + \sqrt{\vt^2 + {1}/{3}}.
\end{equation*}
That is, we have $(0, \overline{\beta}(E)) \in \dD_{\omega, B}(E)$ as required.
\end{proof}

Now we have the main result in this subsection.
\begin{thm}
Conjectures \ref{conj:BMS} and \ref{conj:limitBG} are equivalent.
\end{thm}
\begin{proof}
The
proof is identical to that of \cite[Theorem 5.4]{BMS} as we have  Propositions \ref{prop:BG-zero-delta} and \ref{prop:beta-limit-position} in general setup. 
\end{proof}

\begin{rmk}
Consequently, we have two equivalent forms Conjectures \ref{conj:BMS} and \ref{conj:limitBG}  of  Conjecture \ref{conj:BMT}.
In \cite{BMS}, the authors showed that Conjecture~\ref{conj:BMT} holds for abelian 3-folds. They firstly reduced Conjecture~\ref{conj:BMT} for $B$ and $\omega$ are parallel cases by using the multiplication map in abelian varieties, and then proved Conjecture \ref{conj:limitBG} for those cases. 
Since we have  the equivalent formulations of the conjectures in general, 
following similar arguments in   \cite[Section 7]{BMS}, one can directly prove Conjecture \ref{conj:limitBG} for abelian 3-folds. 
\end{rmk}
\section{Moduli stacks of semistable objects}\label{sec:moduli}
Our main aim of this section is to complete the proof of Theorem~\ref{thm:intro:stack}. 
\subsection{Notation}\label{subsec:state}
Let $X$ be a smooth projective variety. 
Due to Lieblich~\cite{LIE}, 
there is an algebraic stack $\mM$ locally of finite
type  parameterizing objects $E \in D^b \Coh(X)$
with
\begin{align*}
\Ext^{<0}(E, E)=0.
\end{align*}
Let $\sigma=(Z, \aA)$ be a stability condition on 
$D^b \Coh(X)$ with respect to 
some data $(\Gamma, \cl)$.  
For given $v \in \Gamma$, one can consider 
the substack
\begin{align}\label{M:emb}
\mM_{\sigma}(v) \subset \mM
\end{align}
which parametrizes $\sigma$-semistable objects
$E \in \aA$ with $\cl(E)=v$.
A priori, we do not know whether 
$\mM_{\sigma}(v)$ is an algebraic stack 
nor is of finite type. 

Suppose that $\dim X=3$, and let
$B \in \NS(X)_{\mathbb{Q}}$ and $\omega \in \NS(X)_{\mathbb{R}}$ an ample 
class with $\omega^2$ rational as in \eqref{class:choice}. 
If we assume Conjecture~\ref{conj:BMT}, then 
we have the associated Bridgeland 
stability condition $\sigma_{\omega, B}$ given 
by (\ref{BMT:stab}). 
Note that $\sigma_{\omega, B}$ is good in the sense of 
Definition~\ref{def:good} by Remark~\ref{rmk:good}. 
Let
\begin{align*}
\Stab_{\omega, B}^{\circ}(X) \subset \Stab_{\omega, B}(X)
\end{align*}
be the connected 
component which contains $\sigma_{\omega, B}$. 
\begin{rmk}
Indeed, an argument in~\cite[Proposition~8.10]{BMS} shows that 
any $\sigma_{k\omega, B}$ for any 
$k \in \mathbb{Q}_{>0}$ are in the same 
connected component $\Stab_{\omega, B}^{\circ}(X)$. 
However, we will not use this fact. 
\end{rmk}

The purpose of this section is to prove the 
following result: 
\begin{thm}\label{thm:stack}
Let $X$ be a smooth projective 3-fold
satisfying 
Conjecture~\ref{conj:BMT}. 
Then for any $v \in \Gamma_{\omega, B}$
and 
$\sigma \in \Stab_{\omega, B}^{\circ}(X)$, 
the stack 
$\mM_{\sigma}(v)$
is a quasi-proper algebraic stack of finite type over $\mathbb{C}$,
such that the embedding (\ref{M:emb}) is an open immersion. 
\end{thm}
Here 
an algebraic stack is called quasi-proper if 
it satisfies the 
valuative criterion of properness
in~\cite[Definition~2.44]{GT}, without the separatedness
condition.  
The proof of the above Theorem will 
be given in Subsection~\ref{subsec:proof}. 
The key ingredients are the 
boundedness and the 
generic flatness statements. 
First we recall the boundedness.  
\begin{defi}
A set of isomorphism classes of 
objects $\sS$ in $D^b \Coh(X)$
is called \textit{bounded}
if there is a $\mathbb{C}$-scheme $S$ of 
finite type, and an object 
\begin{align*}
\eE \in D^b \Coh(X \times S)
\end{align*}
 such that any object in $\sS$ is isomorphic to 
$\eE_s \cneq \dL i_{s}^{\ast}\eE$ for some $s \in S$. 
Here $i_s \colon X \times \{s \} \hookrightarrow X \times S$ is the inclusion. 
\end{defi}

Next we recall the generic flatness. 
Let $S$ be a smooth projective variety, and $\lL$ an ample line 
bundle on $S$. 
Let $\aA \subset D^b \Coh(X)$ be the heart of a bounded t-structure on 
$D^b \Coh(X)$ which is noetherian. 
By~\cite[Theorem~2.6.1]{AP}, the category
\begin{align}\label{def:AS}
\aA_{S} \cneq \{\eE \in D^b \Coh(X\times S) : 
\dR p_{X \ast}(\eE \otimes p_S^{\ast}\lL^{\otimes m}) \in \aA, \ m \gg 0 \}
\end{align}
is the heart of a bounded t-structure on $D^b \Coh(X \times S)$. 
Here $p_X$, $p_S$ are the projections from $X \times S$ onto 
the corresponding factors. 
\begin{defi}
We say that $\aA$ satisfies the generic flatness
if for any smooth projective variety $S$ and 
an object $\eE \in \aA_S$, there is a non-empty 
open subset $U \subset S$ such that 
$\eE_s \in \aA$ for any $s\in U$. 
\end{defi}
Note that if $\sigma=(Z, \aA)$ is a good stability condition, 
then the heart $\aA$ is noetherian (see~Remark~\ref{rmk:good}). 
We say that a good stability condition 
$\sigma=(Z, \aA)$ satisfies the generic flatness
if the heart $\aA$ satisfies the generic flatness.  

Our strategy for Theorem~\ref{thm:stack}
is to show that 
the boundedness and the generic flatness
are preserved (in some sense) 
under the tilting of very weak stability 
conditions 
$(Z, \aA) \leadsto (Z_t^{\dag}, \aA_t^{\dag})$
for $t>0$. 
This argument shows 
the required results for the stability 
condition 
$\sigma_{\omega, B}$ (see~Corollary~\ref{cor:Bbound}). 
Then we show that these properties are preserved under
the deformations of a stability condition (see~Proposition~\ref{lem:flat:con}),
and together with Abramovich-Polishchuk's valuative criterion
(see~Theorem~\ref{thm:val}) we 
conclude the result.

\subsection{Induction argument for the boundedness}
Let $X$ be a smooth projective variety, 
and take $\dD=D^b \Coh(X)$. 
We fix the data $(\Gamma, \cl)$ 
to consider very weak stability conditions on $\dD$, 
and use the notation as in Section~\ref{sec:veryweak}. 
Let $(Z, \aA)$ be a very weak stability condition 
on $\dD$ satisfying a BG inequality and which is also good
(see~Definition~\ref{def:good}). 
For $t>0$, let
 $(Z_t^{\dag}, \aA^{\dag})$ be the 
associated tilting given in 
Corollary~\ref{cor:support}. 
The purpose of this subsection is 
to prove the boundedness of 
$\mu_t^{\dag}$-semistable 
objects in $\aA^{\dag}$, 
assuming some kind of 
boundedness for $\mu$-semistable objects. 
We first prepare some notation. 
Let us fix an isomorphism
\begin{align}\label{fix:isom}
\Gamma_0 \otimes_{\mathbb{Z}} \mathbb{Q} \stackrel{\cong}{\to}
\mathbb{Q}^{r}.
\end{align}
\begin{defi}
We say a subset $S \subset \mathbb{Q}^r$ is 
\textit{bounded below (resp. above)} if 
there exist functions $f_i(x_1, \cdots, x_{i-1})$
for each $1\le i\le r$ such that 
every element $(s_1, \cdots, s_r) \in S$
satisfies
\begin{align*}
s_i \ge f_i(s_1, \cdots, s_{i-1}), \ 
(\mbox{resp. } s_i \le f_i(s_1, \cdots, s_{i-1})). 
\end{align*}
\end{defi}
\begin{rmk}\label{rmk:bounded}
It is easy to see that a subset $S \subset \mathbb{Q}^r$ is bounded
below (above) if and only if $S \cap S'$ is a finite 
set for any bounded above (below) subset
$S' \subset \mathbb{Q}^r$. 
\end{rmk}
If we choose an isomorphism (\ref{fix:isom}), 
it defines the notion of bounded 
below (above) subsets in $\Gamma_0$. 
For $v\in \Gamma$
and a subset $S \subset \Gamma_0$, let 
$M_{\mu}(v, S)$ be the 
set of isomorphism classes of 
$\mu$-semistable objects $E \in \aA$ with 
$\cl(E) \in v+S$. 
\begin{defi}\label{defi:boundZA}
A very weak stability condition $(Z, \aA)$
satisfies boundedness if either 
(i) or (ii) holds: 
\begin{enumerate}
\item $\Gamma_0 \neq 0$, and under a
suitable isomorphism (\ref{fix:isom}),
the set $M_{\mu}(v, S)$ of isomorphism classes of objects 
is bounded for any $v\in \Gamma$
with $\mu(v)<\infty$ and any bounded 
below subset $S \subset \Gamma_0$. 
\item $\Gamma_0=0$ and the same condition
of (i) holds for any $v\in \Gamma$. 
\end{enumerate}
\end{defi}

We put the following assumption
(recall the category $\pP_0^{\dag}(\phi)$ in 
Lemma~\ref{lem:t=0} and the 
slope function $\lambda$ in (\ref{def:lambda})): 
\begin{assum}\label{assum:bound}
\begin{enumerate}

\item 
The very weak stability condition $(Z, \aA)$
satisfies boundedness. 

\item For any $v \in \Gamma$
with $\mu(v) <\infty$
and any bounded below subset $S^{\dag} \subset \Gamma_0^{\dag}$, 
the isomorphism classes
of objects
\begin{align*}
E \in \langle U[1], \cC : U \in \aA \mbox{ is } \mu \mbox{-semistable} \rangle
\end{align*}
satisfying $\Hom(\cC, E)=0$ and $\cl(E) \in v+S^{\dag}$
is bounded. 

\item For any $v \in \Gamma \setminus \Gamma_0$, 
and any bounded below subset $S^{\dag} \subset \Gamma_0^{\dag}$, the 
isomorphism classes of $\lambda|_{\pP^{\dag}_0(1/2)}$-semistable 
objects $E \in \pP^{\dag}_0(1/2)$ with $\cl(E) \in v+S^{\dag}$
is bounded. 
\end{enumerate}
\end{assum} 
We have the following proposition: 
\begin{prop}\label{prop:bound}
Suppose that Assumption~\ref{assum:bound} holds. 
If $\Gamma_0^{\dag} \neq 0$, then  
the very weak stability condition $(Z_t^{\dag}, \aA^{\dag})$
satisfies boundedness for any $t>0$.  
\end{prop}
\begin{proof}
Let us take $v\in \Gamma$ with 
$\mu_t^{\dag}(v)<\infty$, i.e. 
$\Ree Z(v) \neq 0$, and a bounded 
below subset $S^{\dag} \subset \Gamma_0^{\dag}$. 
We denote by $M_t^{\dag}(v, S^{\dag})$
the set of isomorphism classes of 
$\mu_t^{\dag}$-semistable 
objects $E \in \aA^{\dag}$ with $\cl(E) \in v+S^{\dag}$. 
We show the boundedness of 
$M_t^{\dag}(v, S^{\dag})$
by induction on 
$\Imm Z_t^{\dag}(v)=-\Ree Z(v)>0$. 
We set $R>0$ as in (\ref{min}), 
and suppose that $-\Ree Z(v)=R$.
Then any $E \in M_t^{\dag}(v, S^{\dag})$ is $\mu_{t'}^{\dag}$-semistable 
for $0<t'\ll 1$. 
If $0<\mu(v)<\infty$, 
then $E\in \aA$
and it is $\mu$-semistable 
by Proposition~\ref{lem:tsmall}. 
Hence, the 
boundedness of $M_t^{\dag}(v, S^{\dag})$
follows from Assumption~\ref{assum:bound} (i). 
If $\mu(v) \le 0$, 
by Proposition~\ref{lem:tsmall}, any object $E\in M_t^{\dag}(v, S^{\dag})$ 
fits into the
exact sequence in $\aA^{\dag}$
\begin{align*}
0\to
U[1] \to E \to F \to 0
\end{align*}
where $F \in \cC$
and $U\in \aA$ is a non-zero $\mu$-semistable object with $\mu(U)=\mu(v)$. 
Note that $\Hom(\cC, E)=0$
as $E$ is $\mu_t^{\dag}$-semistable with $\mu_t^{\dag}(E)<\infty$
and $\mu_t^{\dag}(T)=\infty$ for any $T \in \cC$. 
Hence, the boundedness of $M_t^{\dag}(v, S^{\dag})$
follows from Assumption~\ref{assum:bound} (ii). 
If $\mu(v)=\infty$, 
then $E \in \pP^{\dag}_0(1/2)$, 
and it is $\lambda|_{\pP^{\dag}_0(1/2)}$-semistable
by Proposition~\ref{lem:tsmall}. 
Hence, the boundedness of $M_t^{\dag}(v, S^{\dag})$
 follows from Assumption~\ref{assum:bound} (iii). 

Suppose that $-\Ree Z(v)>R$.
Let $U \subset \mathbb{R}_{>0}$ be the 
set of $t\in \mathbb{R}_{>0}$ 
such that $M_t^{\dag}(v, S^{\dag})$ is bounded. 
It is enough to show that $U$
is non-empty, open and closed. 
The same argument as in the case of 
$-\Ree Z(v)=R$
together with 
Lemma~\ref{lem:wallemp} and 
Proposition~\ref{lem:tsmall} 
show
that $(0, t_0) \subset U$
for some $t_0>0$. 
In particular, $U$ is non-empty. 
Also as in the proof of Lemma~\ref{lem:wallemp}, 
there is a locally finite set of walls 
$W \subset \mathbb{R}_{>0}$ such that 
$M_t^{\dag}(v, S^{\dag})$ is constant on 
each connected component of $\mathbb{R}_{>0} \setminus W$, 
which  
implies that  
$U$ is open. 
In order to show that $U$ is closed, 
we take a chamber $V \subset \mathbb{R}_{>0} \setminus W$
satisfying $V \subset U$
and take $t \in \overline{V} \cap W$. 
It is enough to show that $t\in U$. 
Let $M_t^{s \dag}(v, S^{\dag})$ be the subset of $M_t^{\dag}(v, S^{\dag})$
consisting of $\mu_t^{\dag}$-stable objects. 
Then $M_t^{s \dag}(v, S^{\dag}) \subset M_{t'}^{\dag}(v, S^{\dag})$
for $t' \in V$, and so $M_t^{s \dag}(v, S^{\dag})$ is bounded. 
If $E \in M_t^{\dag}(v, S^{\dag})$ is not $\mu_t^{\dag}$-stable, 
then there is an exact sequence 
\begin{align*}
0 \to E_1 \to E \to E_2 \to 0
\end{align*}
in $\aA^{\dag}$
with $\mu_t^{\dag}(E_1)=\mu_t^{\dag}(E_2)$. 
By the support property of $(Z_t^{\dag}, \aA^{\dag})$, 
the set of possible vectors
$\cl(E_i) \in \Gamma/\Gamma_0^{\dag}$ 
is bounded. Also $0<-\Ree Z(E_i)<-\Ree Z(v)$, and so 
the induction hypothesis is applied. 
Noting Remark~\ref{rmk:bounded}, it 
follows that there is a finite 
subset $S' \subset \Gamma/\Gamma_0^{\dag}$
and a bounded above subset $S'' \subset \Gamma_0^{\dag}$
such that all the possible $E_i$ 
are objects in $M_t^{\dag}(v', S'')$ for some $v' \in S'$. 
As $S^{\dag}$ is bounded below, one can take $S''$ to be a finite 
set, and so $M_t^{\dag}(v', S'')$ is bounded by the induction hypothesis. 
Therefore, $M_t^{\dag}(v, S^{\dag})$ is also bounded, and $t\in U$ holds. 
\end{proof}
When $(Z_t^{\dag}, \aA^{\dag})$ becomes a stability condition, 
we have the following result on the boundedness: 
\begin{lem}\label{prop:bound2}
Suppose that Assumption~\ref{assum:bound} holds, 
$\Gamma_0^{\dag}=0$ and the 
set of $F \in \cC$ with fixed $\cl(F)$ is bounded. 
Then
the stability condition $(Z_t^{\dag}, \aA^{\dag})$
satisfies boundedness for any $t>0$.  
\end{lem}
\begin{proof}
By Proposition~\ref{prop:bound}, we only need to 
show the boundedness of $\mu_t^{\dag}$-semistable objects
$E \in \aA^{\dag}$ with a fixed $\cl(E) = v \in \Gamma$
satisfying $\mu_t^{\dag}(v)=\infty$. 
Any such object $E$ fits into a short exact sequence in $\aA^{\dag}$
\begin{align*}
0 \to U[1] \to E \to F \to 0
\end{align*}
for some $\mu$-semistable $U \in \aA$ with 
$\mu(U)=0$, and $F \in \cC$. 
From the support property of $(Z_t^{\dag}, \aA^{\dag})$, 
$\cl(U)$ and $\cl(F)$ are bounded. 
By Assumption~\ref{assum:bound} (i) and the boundedness of 
$F$, we have the boundedness of possible $E$. 
\end{proof}

\subsection{Generic flatness}
We carry on the setting in the previous subsection. 
Here we discuss the generic flatness under tilting. 
The proof of the following proposition is almost the same 
as~\cite[Proposition~3.18]{Tst3}, but we include the proof for the
reader's convenience.
\begin{prop}\label{prop:generic}
Let $(Z, \aA)$ be a 
very weak stability condition on $\dD$
which satisfies a BG inequality and is good. 
If $(Z, \aA)$ satisfies Assumption~\ref{assum:bound}
and $\aA$ satisfies the generic flatness,
then $\aA^{\dag}$ satisfies the generic flatness. 
\end{prop}
\begin{proof}
Let $S$ be a smooth projective variety
and $\lL$ an ample line bundle on it. 
Since the category $\aA^{\dag}$ is noetherian
(see~Lemma~\ref{lem:noether}), 
a similar construction of (\ref{def:AS}) defines the 
heart 
\begin{align*}
\aA_{S}^{\dag} \subset D^b \Coh(X \times S).
\end{align*}
Let us take $\eE \in \aA_S^{\dag}$. 
By the definition of $\aA_S$ and $\aA_S^{\dag}$, we have
\begin{align*}
&\dR p_{X \ast}(\eE \otimes p_S^{\ast}\lL^{\otimes m}) \in \aA^{\dag}
 \subset \langle \aA[1], \aA \rangle \\
&\dR p_{X \ast}(\hH_{\aA_S}^i(\eE) \otimes p_S^{\ast}\lL^{\otimes m}) \in \aA
\end{align*}
for any $i \in \mathbb{Z}$ and $m\gg 0$.  
Therefore, we have
\begin{align*}
\dR p_{X \ast}(\hH_{\aA_S}^i(\eE) \otimes p_S^{\ast}\lL^{\otimes m})=0
\end{align*}
for any $i\neq 0, -1$ and $m\gg 0$. 
This implies that $\hH_{\aA_S}^i(\eE)=0$ for $i\neq 0, -1$. 
By the assumption of the 
generic flatness of $\aA_S$, 
there is a non-empty open subset $U \subset S$ such that  
$\hH_{\aA_S}^i(\eE)_s \in \aA$ for any $s \in U$ and $i \in \{-1, 0\}$. 
It remains to show that
there is a non-empty open subset $U' \subset U$
such that 
\begin{align*}
\hH_{\aA_S}^{0}(\eE)_s \in \pP((1/2, 1]), \ 
\hH_{\aA_S}^{-1}(\eE)_s \in \pP((0, 1/2]). 
\end{align*}
for any $s\in U'$. 
The condition $\hH_{\aA_S}^0(\eE)_s \notin \pP((1/2, 1])$ 
is equivalent to the existence of 
a surjection $\hH_{\aA_S}^0(\eE)_s \twoheadrightarrow Q$
in $\aA$ with $\mu(Q)\le 0$. 
By Assumption~\ref{assum:bound}, 
for any $v \in \Gamma$
with $\mu(v) <\infty$, 
the set of $\mu$-semistable objects 
$E \in \aA$ with $\cl(E)=v$ is bounded. 
Hence, by~\cite[Proposition~3.17]{Tst3} 
there is a $\mathbb{C}$-scheme of finite type 
\begin{align*}
\pi_{\qQ} \colon 
\qQ uot(\hH_{\aA_S}^0(\eE)) \to U
\end{align*}
whose fiber at a point $s \in U$ parametrizes
the quotients $\hH_{\aA_S}^0(\eE)_s \twoheadrightarrow Q$
with $\mu(Q)\le 0$. 
On the other hand, by~\cite[Proposition~3.5.3]{AP} 
the set of points $s \in U$ with 
$\eE_s \in \aA^{\dag}$ is dense in $U$. 
This implies that the complement 
$U \setminus \im \pi_{\qQ}$ is 
dense, and because it is constructible, 
there is a non-empty
 open subset $U_1 \subset U \setminus \im \pi_{\qQ}$. 
By the construction of $\pi_{\qQ}$, 
we have $\hH_{\aA_S}^{0}(\eE)_s \in \pP((1/2, 1])$
for any $s\in U_1$. 
One can similarly find a non-empty
open subset $U_2 \subset U$
such that $\hH_{\aA_S}^{-1}(\eE)_s \in \pP((0, 1/2])$
holds for any $s\in U_2$. We obtain the desired $U'$
by setting
$U'=U_1 \cap U_2$.
\end{proof}

\subsection{Boundedness and generic flatness under deformations}
In the setting of the previous subsection, we show that the boundedness and the generic flatness are preserved
under the deformations of stability conditions. 
The following statement is stronger than~\cite[Theorem~3.20]{Tst3}, 
but the essential arguments were already there. 
\begin{prop}\label{lem:flat:con}
Let $\Stab_{\Gamma}^{\circ}(\dD) \subset \Stab_{\Gamma}(\dD)$
be a connected component. 
Suppose that 
there is a good stability condition
$\sigma \in \Stab_{\Gamma}^{\circ}(\dD)$
satisfying the boundedness
and the generic flatness. 
Then any good stability 
condition $\tau \in \Stab_{\Gamma}^{\circ}(\dD)$
satisfies boundedness and the generic flatness. 
\end{prop}
\begin{proof}
Let $\{\pP(\phi)\}_{\phi \in \mathbb{R}}$ and 
$\{\qQ(\phi)\}_{\phi \in \mathbb{R}}$
be the slicing 
determined by $\sigma$ and $\tau$ respectively. 
By connecting $\sigma$ and $\tau$ via a path
and cutting it into several pieces, 
we may assume that 
\begin{align*}
\qQ(\phi) \subset 
\pP \left(\left(\phi-\varepsilon, \phi+\varepsilon
 \right) \right)
\end{align*}
holds for any $\phi \in \mathbb{R}$ and some $0<\varepsilon <1/8$.  
Then the argument of~\cite[Theorem~3.20, Step~1]{Tst3}
shows that, for a fixed $v\in \Gamma$, 
the stack $\mM_{\tau}(v)$
is an 
algebraic substack of finite type 
such that $\mM_{\tau}(v) \subset \mM$
is an open immersion.  
In particular, $\tau$ satisfies boundedness. 
Indeed, one can replace $(\sigma, \sigma')$ 
in the proof of~\cite[Theorem~3.20, Step~1]{Tst3}
by $(\tau, \sigma)$, 
 and the rest of the arguments are the same. 

It remains to show the 
generic flatness of $\tau$. 
Let $\bB$ be the heart $\qQ((0, 1])$, and take 
an object $\eE \in \bB_S$ 
for a smooth projective variety $S$. 
We set
\begin{align*}
S' \cneq \{ s\in S : 
\eE_s \in \bB\}. 
\end{align*}
We need to show that $S'$ contains an open 
subset of $S$. 
By the result of~\cite[Lemma~2.6.2]{AP}, there is an object
$G \in \bB$, an ample line bundle $\lL$
on $S$ and a surjection 
$G \boxtimes \lL^{-1} \twoheadrightarrow \eE$
in $\bB_S$. 
By~\cite[Lemma~2.5.7]{AP}, 
the functor 
\begin{align*}
\dL i_s^{\ast} \colon 
D^b \Coh(X \times S) \to D^b \Coh(X)
\end{align*}
for any $s\in S$
is right exact
with respect to the t-structures with hearts
$\bB_S$ and $\bB$ respectively. 
Therefore,  
we have the surjection $G \twoheadrightarrow \eE_s$
in $\bB$
for $s\in S'$. 
Since there is $\psi \in (0, 1]$ such that 
$G \in \qQ((\psi, 1])$, we have 
$\eE_s \in \qQ((\psi, 1])$ for any $s\in S'$. 
By the support property
of $\tau$,  
there is only a finite number 
of ways to write $\cl(F)$
for an object $F \in \qQ((\psi, 1])$ as 
$v_1+\cdots +v_l$ for
$v_i \in \Gamma$ of the form $\cl(F_i)$ 
for some $F_i \in \qQ(\phi_i)$ with $\phi_i \in (\psi, 1]$.
 As we proved above, the stack of 
objects $F_i \in \qQ(\phi_i)$ with 
fixed $v^B(F_i)=v_i$ is 
an algebraic stack of finite type. Hence,
the set of closed points of the stack
\begin{align}\label{Ppsi1}
\oO bj \left(\qQ((\psi, 1])\right) \subset \mM
\end{align}
of objects in $\qQ((\psi, 1])$ is 
locally constructible. 
Therefore, the set of points $S'$ 
is constructible. 
On the other hand, $S'$ is dense
in $S$ by~\cite[Proposition~3.5.3]{AP}. 
Therefore, $S'$ contains an open subset. 
\end{proof}

\subsection{Boundedness of tilt semistable objects}
Below we assume that 
$X$ is a smooth projective 3-fold, and take 
$v^B$, $\Gamma=\Gamma_{\omega, B}$ 
and $Z$ as in Subsection~\ref{subsec:slope}.
We have the
very weak stability condition $(Z, \Coh(X))$ 
on $D^b \Coh(X)$, and let $\mu$ be the associated 
slope function (\ref{tslope}). 
In this setting, the category $\cC$ is given by 
$\Coh_{\le 1}(X)$
and $\cC^{\dag}$ is given by $\Coh_0(X)$. 
The subgroup 
$\Gamma_0 \subset \Gamma$ is given by $v_0^B=v_1^B=0$, 
and $\Gamma_0^{\dag} \subset \Gamma$ is
given by $v_0^B=v_1^B=v_2^B=0$. 
We have the natural isomorphism
\begin{align*}
(v_2^B, v_3^B) \colon \Gamma_0 \otimes_{\mathbb{Z}}
\mathbb{Q} \stackrel{\sim}{\to}
\mathbb{Q}^2. 
\end{align*}
We repeatedly use the following result by Langer: 
\begin{thm}\emph{(\cite[Theorem~4.4]{Langer})}\label{thm:Langer}
(i) 
The set of isomorphism classes of $\mu_{\omega, B}$-semistable 
$E\in \Coh(X)$ with fixed $(v_0^B, v_1^B, v_2^B)$ and bounded 
below $v_3^B$, is bounded. 

(ii) The set of isomorphism classes of $\widehat{\mu}_{\omega, B}$-semistable two dimensional $E \in \Coh_{\le 2}(X)$
with fixed $(v_1^B, v_2^B)$ and bounded below $v_3^B$, is bounded. 
\end{thm}
Here the $\widehat{\mu}_{\omega, B}$-semistability 
on $\Coh_{\le 2}(X)$ is given by the slope function 
$v_2^B(\ast)/v_1^B(\ast)$. 
Using the above result, we prove the following 
proposition: 
\begin{prop}\label{prop:tbound}
The very weak stability condition 
$(Z, \Coh(X))$ satisfies Assumption~\ref{assum:bound}. 
\end{prop}
\begin{proof}
The condition (i)
in Assumption~\ref{assum:bound}
follows from Theorem~\ref{thm:Langer} (i)
together with the classical BG inequality 
$(v_1^B)^2 \ge 2v_0^B v_2^B$ for torsion 
free semistable sheaves. 
The condition (iii)
follows from Theorem~\ref{thm:Langer} (ii). 
The condition (ii) follows from Lemma~\ref{lem:boundC}
below.  

\end{proof}
In order to show (ii) in the proof of Proposition~\ref{prop:tbound}, 
we prepare
some notation and a lemma. 
For $\vartheta \in \mathbb{R}$, we 
set
\begin{align*}
\cC_{\vartheta} \cneq \langle U[1], \Coh_{\le 1}(X) : 
U \mbox{ is }\mu_{\omega} \mbox{-semistable with }
\mu_{\omega}(U)=\vartheta \rangle. 
\end{align*}
\begin{rmk}
The category $\cC_{\vartheta}$
consists of objects $E \in \bB_{\omega, \vartheta\omega}$  with
$v_1^{\vartheta\omega}(E)=0$.  
In particular, $\cC_{\vartheta}$ is an abelian subcategory
of $\bB_{\omega, \vartheta\omega}$. 
\end{rmk}
We  define the following subcategories of $\cC_{\vartheta}$
\begin{align}\label{def:Csharp}
&\cC_{\vartheta}^{\sharp} \cneq
\{ E \in \cC_{\vartheta} :
\Hom(\Coh_{\le 1}(X), E)=0 \} \\
&\cC_{\vartheta}^{\flat} \cneq 
\{ E \in \cC_{\vartheta} : 
\hH^0(E) \in \Coh_0(X), \Hom(\Coh_0(X), E)=0\}. 
\end{align}
Let $\mathbb{D}$ be the dualizing functor 
defined by $\dR \hH om(-, \oO_X)[2]$. 
We have the following lemma: 
\begin{lem}\label{lem:dual}
We have the anti-equivalence of categories
\begin{align*}
\mathbb{D} \colon \cC_{\vartheta}^{\sharp} \stackrel{\sim}{\to}
\cC_{-\vartheta}^{\flat}. 
\end{align*}
\end{lem}
\begin{proof}
For an object $E \in \cC_{\vartheta}^{\sharp}$,
the proof of~\cite[Lemma~3.8]{TodBG}
shows that $\mathbb{D}(E) \in \cC_{-\vartheta}^{\flat}$. 
We only show that $\mathbb{D}(E) \in \cC_{\vartheta}^{\sharp}$
for $E \in \cC_{-\vartheta}^{\flat}$. 
We have an exact sequence in $\cC_{-\vartheta}$
\begin{align}\label{UEF:ex}
0 \to U[1] \to E \to Q \to 0
\end{align}
for a $\mu_{\omega}$-semistable sheaf 
$U$ with $\mu_{\omega}(U)=-\vartheta$
and $Q \in \Coh_0(X)$. 
Applying $\mathbb{D}$
and taking the long exact sequence of cohomologies with respect to $\Coh(X)$, 
we obtain the 
isomorphisms
$\hH^{i} \mathbb{D}(E) \cong 0$
for $i\neq -1, 0, 1$, 
 $\hH^{-1}\mathbb{D}(E) \cong U^{\vee}$ and the 
exact sequence of sheaves
\begin{align*}
0 \to \hH^0\mathbb{D}(E) \to \eE xt_X^1(U, \oO_X) \to 
Q^{\vee} \to \hH^1 \mathbb{D}(E) \to \eE xt^2(U, \oO_X) \to 0.
\end{align*}
In particular, $\hH^0 \mathbb{D}(E)$ is at most one dimensional, 
and $\hH^1 \mathbb{D}(E)$ is zero dimensional. 
By dualizing $\Hom(\Coh_0(X), E)=0$, 
we have $\hH^1\mathbb{D}(E)=0$, and so
$\mathbb{D}(E)\in \cC_{\vartheta}$. 
It remains to show $\Hom(\Coh_{\le 1}(X), \mathbb{D}(E))=0$.
By dualizing, this is equivalent to 
$\Hom(E, F)=0$ for any pure one dimensional sheaf $F$
and $\Hom(E, \Coh_0(X)[-1])=0$. 
These conditions obviously follow from the condition
$E\in \cC_{-\vartheta}^{\flat}$. 
\end{proof}
\begin{lem}\label{lem:boundC}
For any bounded below subset $S^{\dag} \subset \Gamma_0^{\dag}$
and $v\in \Gamma$, 
the set of $E \in \cC_{\vartheta}^{\sharp}$ 
with $v^B(E) \in v+S^{\dag}$ is bounded. 
\end{lem}
\begin{proof}
By Lemma~\ref{lem:dual}, 
it is enough to show the boundedness of 
$E \in \cC_{\vartheta}^{\flat}$
with $v^B(E) \in v+S'$
for a bounded \textit{above} subset $S' \subset \Gamma_0^{\dag}$. 
By Theorem~\ref{thm:Langer} (i), 
the set of possible $v^B_3(\hH^{-1}(E)[1])$ is bounded 
below, and we have 
$v^B_3(\hH^0(E)) \ge 0$
since $\hH^0(E)$ is zero dimensional.  
Hence, one can take $S'$ to be a finite set, 
and the result follows using
Theorem~\ref{thm:Langer} (i) again. 
\end{proof}

Applying Proposition~\ref{prop:bound} to 
$(Z, \Coh(X))$, we have the following corollary: 
\begin{cor}\label{cor:btilt}
The set of $\nu_{\omega, B}$-semistable objects
$E \in \bB_{\omega, B}$ with 
fixed $(v_0^B, v_1^B, v_2^B)$
satisfying $v_1^B>0$,
and bounded below $v_3^B$, is bounded. 
\end{cor}

\subsection{Boundedness of Bridgeland semistable objects}
We carry on the notation in the previous subsection. 
Let 
$W$ be the group homomorphism given by (\ref{Ztilt}). 
We have the very weak stability condition $(W, \bB_{\omega, B})$
studied in Subsection~\ref{subsec:double}, 
and denote by $\{\qQ(\phi)\}_{\phi \in \mathbb{R}}$ the
slicing determined by $(W, \bB_{\omega, B})$. 
For $t>0$, let 
$(W_t^{\dag}, \bB_{\omega, B}^{\dag})$ be 
its tilting given in (\ref{dtilt}), and $\nu_t^{\dag}$, 
$\{\qQ_t^{\dag}(\phi)\}_{\phi \in \mathbb{R}}$
 the 
 slope function, slicing determined by $(W_t^{\dag}, \bB_{\omega, B}^{\dag})$,
respectively. 
By the argument in Subsection~\ref{subsec:to0}, 
we have the slope function on $\qQ_t^{\dag}(\phi)$
\begin{align}\label{xi}
\xi \cneq \frac{\Delta_I^{\dag}}{\Ree W}
=\frac{18v_3^B-v_1^B}{6v_2^B-v_0^B}
\end{align}
where $\Delta_I^{\dag}$ is given by (\ref{Delta3}).
In this notation, 
we have the following lemma: 
\begin{lem}\label{lem:boundP}
For a fixed $v\in \Gamma$, 
the set of $\xi|_{\qQ^{\dag}_0(1/2)}$-semistable 
objects $E \in \qQ^{\dag}_0(1/2)$ with $v^B(E)=v$ is bounded. 
\end{lem}
\begin{proof}
Note that the category $\qQ(1)$ is given by
\begin{align*}
\qQ(1)=\langle U[1], \Coh_{\le 1}(X) : U \mbox{ is } \mu_{\omega, B}
\mbox{-semistable with }\mu_{\omega, B}(U)=0 \rangle 
\end{align*}
i.e. it coincides with $\cC_{\vartheta=(B\omega^2)/(\omega^3)}$.
By Lemma~\ref{lem:t=0}, 
the category $\qQ^{\dag}_0(1/2)$ 
consists of 
$\xi$-semistable objects 
$E \in \cC_{(B\omega^2)/(\omega^3)}$
satisfying $\Hom(\Coh_0(X), E)=0$. 
The restriction of $\xi$ to $\qQ^{\dag}_0(1/2)$ is given 
by
\begin{align}\label{xi2}
\xi|_{\qQ^{\dag}_0(1/2)}=
\frac{18v_3^B}{6v_2^B-v_0^B}.  
\end{align}
We let $v \in \Gamma$
be such that 
$v=v^B(E)$
for some $\xi|_{\qQ^{\dag}_0(1/2)}$-semistable 
object $E \in \qQ^{\dag}_0(1/2)$. 
Then we have $v_1^B=0$, 
$v_2^B \ge 0$ and $v_0^B \le 0$. 
There is an exact sequence in $\qQ^{\dag}_0(1/2)$
\begin{align*}
0 \to F \to E \to E' \to 0
\end{align*}
where $F \in \Coh_{\le 1}(X)$
and $E' \in \cC_{(B\omega^2)/(\omega^3)}^{\sharp}$. 
As $v_2^B(\ast)$ is non-negative on 
$\qQ^{\dag}_0(1/2)$, the possible $v_2^B(E')$ are bounded. 
Combined with 
the $\xi|_{\qQ^{\dag}_0(1/2)}$-semistability of $E$, 
we obtain the lower bound of $v_3^B(E')$, and so
the possible $E'$ are bounded by Lemma~\ref{lem:boundC}. 
It remains to show the boundedness of possible $F$. 
Note that the $\xi|_{\Coh_{\le 1}(X)}$-stability 
on $\Coh_{\le 1}(X)$ coincides with the $B$-twisted 
$\omega$-Gieseker stability on it. 
Let $F_1, \cdots, F_k$ be the HN factors of 
$F$ with respect to the $\xi|_{\Coh_{\le 1}(X)}$-stability
such that $F_1$ has the maximal $\xi$-slope. 
Since $F_1 \subset E$ in $\qQ^{\dag}_0(1/2)$, 
the $\xi|_{\qQ^{\dag}_0(1/2)}$-semistability of 
$E$ implies the upper bound of the $\xi$-slope of $F_1$. 
Because the $\xi$-slope of $F_1$ is also bounded below 
$\xi(F_1)\ge \xi(F)$, 
and $v_2^B(F_1)$ is bounded, 
the set of possible $v_3^B(F_1)$ is 
also bounded. 
The same argument shows that
the possible $k\ge 1$ together with 
$v^B(F_i)$ are bounded for all $1\le i\le k$. 
By the 
boundedness of $\xi|_{\Coh_{\le 1}(X)}$-semistable
objects in $\Coh_{\le 1}(X)$ with 
fixed $v_2^B$ and $v_3^B$, we have the boundedness of possible $F_i$, 
and so the boundedness of possible $F$. 
\end{proof}

\begin{prop}
Suppose that $X$ satisfies 
Conjecture~\ref{conj:BMT}. Then 
the very weak stability condition $(W, \bB_{\omega, B})$
satisfies Assumption~\ref{assum:bound}. 
\end{prop}
\begin{proof}
In the notation of the proof of Proposition~\ref{prop:good}, 
we have $\cC^{\dag}=\Coh_0(X)$, $\cC^{\dag \dag}=\{0\}$, and so 
$\Gamma_0^{\dag \dag}=\{0\}$. 
Therefore, the conditions (i), (ii) in Assumption~\ref{assum:bound}
follow
from Corollary~\ref{cor:btilt}. 
The condition (iii) in Assumption~\ref{assum:bound}
follows from Lemma~\ref{lem:boundP}. 
\end{proof}
By combining the results so far, we obtain the 
following result which is required to prove Theorem~\ref{thm:stack}:
\begin{cor}\label{cor:Bbound}
Suppose that $X$ satisfies
 Conjecture~\ref{conj:BMT}. 
Then any good 
stability condition $\sigma \in \Stab_{\omega, B}^{\circ}(X)$
satisfies boundedness
and the generic flatness.   
\end{cor}
\begin{proof}
Note that the set of $F \in \Coh_0(X)$ with fixed $v_3^B(F)$ is bounded. 
Applying 
Lemma~\ref{prop:bound2} to 
$(W, \bB_{\omega, B})$, we 
conclude that $(Z_{\omega, B}, \aA_{\omega, B})$
satisfies boundedness. 
Also by using Proposition~\ref{prop:generic} twice, 
we see that $\aA_{\omega, B}$ satisfies the 
generic flatness. 
The result now follows from Proposition~\ref{lem:flat:con}.  
\end{proof}

\subsection{Proof of Theorem~\ref{thm:stack}}\label{subsec:proof}

\begin{proof}
In~\cite[Theorem~3.20]{Tst3}, 
the second author gave 
general arguments for the
moduli stacks of Bridgeland semistable 
objects to be
algebraic stacks of finite type. 
First, 
because the good stability conditions form 
a dense subset in $\Stab_{\omega, B}^{\circ}(X)$, 
and the walls are defined over $\mathbb{Q}$, 
one can perturb $\sigma \in \Stab_{\omega, B}^{\circ}(X)$
and may assume that $\sigma$ is algebraic
(see~\cite[Theorem~3.20, Step~3]{Tst3}), 
i.e. $\sigma$ is good 
with $\zeta_i=1$ in 
Definition~\ref{def:good}.  
Next
for an algebraic 
stability condition $\sigma$, the problem is reduced to showing 
the boundedness 
and the generic flatness of $\sigma$, 
which are proved in Corollary~\ref{cor:Bbound}. 
Therefore, by~\cite[Theorem~3.20]{Tst3} 
the stack $\mM_{\sigma}(v)$ is 
an algebraic stack of finite type such that 
(\ref{M:emb}) is an open immersion. 

It remains to show the valuative 
criterion of properness of $\mM_{\sigma}(v)$. 
We write $\sigma=(Z, \aA)$, and 
let $\{\pP(\phi)\}_{\phi \in \mathbb{R}}$ be
the associated slicing. 
Let us take $\phi_0 \in (0, 1]$ such that 
$Z(v) \in \mathbb{R}_{>0} e^{i\pi \phi_0}$. 
Then the rotated stability condition
\begin{align}\label{rotate}
(-e^{-i \pi \phi_0}Z, \pP((\phi_0-1, \phi_0]))
\end{align}
is also algebraic. 
Then the required property follows from
Theorem~\ref{thm:val} (i) 
applied to the stability condition (\ref{rotate}).  
\end{proof}
Here we have used the following result by
Abramovich-Polishchuk~\cite{AP}.
\begin{thm}\emph{(\cite[Theorem~4.1.1]{AP})}\label{thm:val}
Let $\sigma \in \Stab_{\omega, B}^{\circ}(X)$ be a good
stability condition, and $\{\pP(\phi)\}_{\phi \in \mathbb{R}}$
the associated slicing. Let $C$ be a smooth curve, 
$p \in C$ a closed point
and set $U=C \setminus \{p\}$. 

\begin{enumerate}
\item  Every family $F_U$ of objects in $\pP(1)$ over $U$ extends to a family
$F$ of objects in $\pP(1)$. 

\item Let $F_1$ and $F_2$ be families of objects in $\pP(1)$ over $S$, 
and $\phi_U \colon (F_1)_{U} \to (F_2)_{U}$ an isomorphism. 
Then $(F_1)_p$ and $(F_2)_p$ are $S$-equivalent. 
\end{enumerate}
\end{thm}
When all the objects $[E] \in \mM_{\sigma}(v)$ are $\sigma$-stable, 
we also have the following result: 
\begin{cor}\label{prop:space}
Suppose that any object $[E] \in \mM_{\sigma}(v)$
is $\sigma$-stable. Then $\mM_{\sigma}(v)$ is a 
$\mathbb{C}^{\ast}$-gerbe over a 
proper algebraic space $M_{\sigma}(v)$ of 
finite type. 
\end{cor}
\begin{proof}
Let $M$ be the algebraic space of simple 
objects in $D^b \Coh(X)$ constructed by Inaba~\cite{Inaba}, 
and write $\sigma=(Z, \aA)$. 
Similarly to the proof of Theorem~\ref{thm:stack}, 
we may assume that $\sigma$ is good. 
The proof of Theorem~\ref{thm:stack} shows that 
the subset $M_{\sigma}(v) \subset M$
of $\sigma$-stable 
$E \in \aA$ with 
$v^B(E)=v$ is an open 
sub algebraic space. 
The valuative criterion 
in Theorem~\ref{thm:val} (i), (ii) applied to (\ref{rotate})
shows that $M_{\sigma}(v)$ is proper.  
By the assumption, any $[E] \in \mM_{\sigma}(v)$ is 
$\sigma$-stable, and so $\Aut(E)=\mathbb{C}^{\ast}$. 
The natural morphism 
$\mM_{\sigma}(v) \to M_{\sigma}(v)$
gives the desired $\mathbb{C}^{\ast}$-gerbe
structure of $\mM_{\sigma}(v)$. 
\end{proof}

\section{Donaldson-Thomas invariants for Bridgeland semistable objects}\label{sec:DT}
In this section, we mainly use Theorem~\ref{thm:stack} to define 
Donaldson-Thomas invariants counting Bridgeland semistable
objects on Calabi-Yau 3-folds which satisfy the BG inequality conjecture. 
\subsection{Donaldson-Thomas invariants}
Let $X$ be a smooth projective Calabi-Yau 3-fold, i.e. 
\begin{align*}
K_X=0, \ H^1(X, \oO_X)=0. 
\end{align*}
Throughout this section, we assume that 
$X$ satisfies Conjecture~\ref{conj:BMT}. 
So far, the only known examples of Calabi-Yau 3-folds
satisfying Conjecture~\ref{conj:BMT}
are A-type Calabi-Yau 3-folds, that are 
\'etale quotients of abelian 3-folds~\cite{BMS}. 
See \cite{OgSa} for a classification of such Calabi-Yau 3-folds. 
We describe one of such examples. 
\begin{exam}
Let $E_1, E_2, E_3$ be three elliptic curves, 
and set $A=E_1 \times E_2 \times E_3$. 
Let $\tau_i \in E_i$ be 2-torsion elements. 
We define the automorphisms $g_1, g_2$ on $A$
to be
\begin{align*}
&g_1(z_1, z_2, z_3)=(z_1+\tau_1, -z_2, -z_3), \\ 
&g_2(z_1, z_2, z_3)=(-z_1, z_2+\tau_2, -z_3+\tau_3). 
\end{align*}
Then $(g_1, g_2)$ defines the free action of 
$G=(\mathbb{Z}/2\mathbb{Z})^{\oplus 2}$ on $A$, 
and $X=A/G$ is an A-type Calabi-Yau 3-fold. 
\end{exam}

Let us take the classes $B$ and $\omega$ as in \eqref{class:choice}. 
Assuming Conjecture~\ref{conj:BMT}, we have 
the Bridgeland stability condition
$\sigma_{\omega, B}$
given by (\ref{BMT:stab}). 
We take the connected component
\begin{align*}
\Stab_{\omega, B}^{\circ}(X)
\subset \Stab_{\omega, B}(X)
\end{align*}
which contains $\sigma_{\omega, B}$. 
Our goal is 
to construct
a map
\begin{align}\label{map:DT}
\DT_{\ast}(v) \colon 
\Stab_{\omega, B}^{\circ}(X) \to \mathbb{Q}
\end{align}
for each $v\in H^{\ast}(X, \mathbb{Q})$, such 
that $\DT_{\sigma}(v)$ virtually counts 
$\sigma$-semistable objects $E \in D^b \Coh(X)$
with $\ch(E)=v$. 

\subsection{DT invariants via virtual classes}
In some cases, the DT invariants may be
defined along with the original idea by Thomas~\cite{Thom}. 
Let $\Lambda$ be the image of the Chern character map
\begin{align*}
\Lambda \cneq \im \left( \ch \colon 
K(X) \to H^{\ast}(X, \mathbb{Q}) \right). 
\end{align*}
Note that the group homomorphism $v^B$ in (\ref{def:vB})
factors through the Chern character map
\begin{align*}
v^B \colon K(X) \stackrel{\ch}{\to}
 \Lambda \stackrel{\alpha}{\to} \Gamma_{\omega, B}. 
\end{align*}
For $v \in \Lambda$ and $\sigma \in \Stab_{\omega, B}^{\circ}(X)$, 
Theorem~\ref{thm:stack} shows that 
the stack $\mM_{\sigma}(\alpha(v))$ is an 
algebraic stack of finite type. 
The substack 
\begin{align}\label{Mstack}
\mM_{\sigma}(v) \subset \mM_{\sigma}(\alpha(v))
\end{align}
of $\sigma$-semistable objects with Chern character $v$
is an open and closed substack. Hence, it is 
also an algebraic stack of finite type. 

Suppose that any object
$[E] \in \mM_{\sigma}(v)$ is $\sigma$-stable. 
Then by Corollary~\ref{prop:space},
the stack $\mM_{\sigma}(v)$ is a
$\mathbb{C}^{\ast}$-gerbe over 
a proper algebraic space $M_{\sigma}(v)$ of finite type. 
The argument of~\cite{HT2} shows that there is 
a symmetric perfect obstruction theory on $M_{\sigma}(v)$, 
and so the zero dimensional virtual 
class $[M_{\sigma}(v)]^{\rm{vir}}$. 
Since $M_{\sigma}(v)$ is proper and separated, 
we can integrate the 
virtual class and define the DT invariant:
\begin{defi}\label{defi:DT:ori}
Suppose that any object $[E] \in \mM_{\sigma}(v)$
is $\sigma$-stable. Then we define
\begin{align}\label{DTwB2}
\DT_{\sigma}(v)  \cneq 
\int_{[M_{\sigma}(v)]^{\rm{vir}}} 1 \in \mathbb{Z}. 
\end{align}
\end{defi}
Note that, by~\cite{Beh}, the invariant (\ref{DTwB2})
is also described as
the weighted Euler characteristic
\begin{align*}
\DT_{\sigma}(v) =\int_{M_{\sigma}(v)} \chi_{B} \ de 
\cneq \sum_{m\in \mathbb{Z}} m \cdot e(\chi_B^{-1}(m)).
\end{align*}
Here $\chi_B$ is Behrend's constructible
 function on $M_{\sigma}(v)$, and $e$
is the topological Euler characteristic. 
\begin{rmk}
The work of~\cite{BBBJ} shows that 
$M_{\sigma}(v)$ is locally written as a critical locus 
of some algebraic function $f$.
The Behrend function $\chi_B$ is locally 
described by the Euler number of the Milnor fiber of 
$f$. 
\end{rmk}
Suppose that $v$ and $\sigma=\sigma_{\omega, B}$
satisfies the assumption in Definition~\ref{defi:DT:ori}.
Then we have the 
associated invariant 
\begin{align*}
\DT_{\omega, B}(v) \cneq \DT_{\sigma_{\omega, B}}(v). 
\end{align*}
\begin{rmk}
The advantage of defining 
the DT invariant 
via the virtual cycle 
is that its deformation invariance
may be easier to prove. 
Indeed 
let $0 \in \Delta \subset \mathbb{C}$ be a small disc, and 
\begin{align*}
\pi \colon \xX \to \Delta
\end{align*}
a smooth one parameter family of smooth projective Calabi-Yau 3-folds.  
Let $B$ be a $\mathbb{Q}$-divisor on $\xX$, 
and $\omega$ be an $\pi$-ample $\mathbb{R}$-divisor on $\xX$ with 
$\omega^2$ rational.
If one can show 
that fiberwise $\sigma_{\omega_t, B_t}$-stable 
objects on $\xX$ for $t \in \Delta$ is proper over $\Delta$, 
then 
the argument of~\cite{BF} shows 
that 
\begin{align*}
\DT_{\omega_t, B_t}(v_t) \in \mathbb{Z}, \ t \in \Delta
\end{align*}
is independent of $t\in \Delta$
for $v \in \Gamma(\Delta, \dR \pi_{\ast}\mathbb{Q})$. 
However, proving the 
above relative properness
requires a relative version of the valuative criterion in 
Theorem~\ref{thm:val}, 
which seems to require a further work. 
\end{rmk}

In general, there may be strictly semistable 
objects $[E] \in \mM_{\sigma}(v)$.
In that case, 
we are not able to define 
the DT invariant by the virtual cycle at the moment. 
Instead, we will use the Hall algebra and the Behrend function 
following the idea of~\cite{K-S}, ~\cite{JS}.

\subsection{Hall algebras}
Here we use the notation in Subsection~\ref{subsec:state}. 
Let $\aA \subset D^b \Coh(X)$ be 
the heart of a bounded t-structure which is noetherian, 
and satisfies the generic flatness. 
Then the 
stack of objects in $\aA$
\begin{align*}
\oO bj(\aA) \subset \mM
\end{align*}
is realized as an open substack of $\mM$. 
In particular, it is an algebraic stack locally 
of finite type. 
The Hall algebra $H(\aA)$ is 
$\mathbb{Q}$-spanned by the isomorphism 
classes of the symbols 
\begin{align}\label{symbol}
[\rho \colon \xX \to \oO bj(\aA)]
\end{align}
where $\xX$ is an Artin stack of finite type over
$\mathbb{C}$ with 
affine geometric stabilizers and
$\rho$ is a 1-morphism. 
The relation is generated by 
\begin{align}\label{Hall:rel}
[\rho \colon \xX \to \oO bj(\aA)]
\sim [\rho|_{\yY} \colon \yY \to \oO bj(\aA)]
+ [\rho|_{\uU} \colon \uU \to \oO bj(\aA)]
\end{align}
where $\yY \subset \xX$ is a closed substack and 
$\uU\cneq \xX \setminus \yY$. 
There is an associative $\ast$-product
on $H(\aA)$
based on the Ringel-Hall algebra (see~\cite[Section~5.1]{Joy2}).
The unit is given by 
$1=[\Spec \mathbb{C} \to \oO bj(\aA)]$,
which corresponds to $0\in \aA$. 
Also there is a Lie subalgebra 
\begin{align*}
H^{\rm{Lie}}(\aA) \subset H(\aA)
\end{align*}
consisting of elements 
supported on \textit{virtual indecomposable 
objects}. 
See \cite[Section~5.2]{Joy2} for further details
on the definition of $H^{\rm{Lie}}(\aA)$. 

The algebra $H(\aA)$
is graded by $\Lambda$
\begin{align*}
H(\aA)=\bigoplus_{v \in \Lambda} H_v(\aA)
\end{align*}
where $H_v(\aA)$ is 
generated by symbols (\ref{symbol}) 
which factors through 
$\oO bj_v(\aA) \subset \oO bj(\aA)$. 
Here $\oO bj_v(\aA)$ is the stack of 
objects $E \in \aA$
with $\ch(E)=v$. 
By the Riemann-Roch theorem, 
the Euler pairing on $K(X)$ 
descends to the anti-symmetric 
pairing 
\begin{align*}
\chi \colon \Lambda \times \Lambda \to \mathbb{Z}. 
\end{align*}
Let $C(\Lambda)$ be the Lie algebra
\begin{align*}
C(\Lambda) \cneq \bigoplus_{v \in \Lambda}
\mathbb{Q} \cdot c_{v}
\end{align*}
with the bracket given by 
\begin{align}\label{bracket}
[c_{v_1}, c_{v_2}] \cneq (-1)^{\chi(v_1, v_2)}
\chi(v_1, v_2) c_{v_1+v_2}. 
\end{align}
 By~\cite[Theorem~5.12]{JS}, there is a $\Lambda$-graded
linear map 
\begin{align}\label{PiHall}
\Pi \colon H^{\rm{Lie}}(\aA) \to C(\Lambda)
\end{align}
such that if $\xX$ is a $\mathbb{C}^{\ast}$-gerbe over an 
algebraic space $\xX'$, we have 
\begin{align}\notag
\Pi([\rho \colon \xX \to \oO bj_v(\aA)])
=-\left(\sum_{k\in \mathbb{Z}}
k \cdot e(\chi_B^{-1}(k))  \right) c_{v}.
\end{align}
Here $\rho$ is an open immersion
and 
$\chi_B$ is  
Behrend's constructible function
on $\xX'$. 
The map (\ref{PiHall}) is shown 
to be a Lie algebra homomorphism 
if $\aA=\Coh(X)$
by~\cite[Theorem~5.12]{JS}.

\subsection{DT invariants for Bridgeland semistable objects}
Let us take the classes $B, \omega$  
as in the previous section (also as in \eqref{class:choice}).  
Consider a good
stability condition
\begin{align*}
\sigma=(Z, \aA)
\in \Stab_{\omega, B}^{\circ}(X).
\end{align*} 
For each element $v \in \Lambda$, 
the 
stack $\mM_{\sigma}(v)$ in (\ref{Mstack})
determines an element
\begin{align*}
\delta_{\sigma}(v) \cneq [\mM_{\sigma}(v) \subset \oO bj(\aA)] \in H(\aA). 
\end{align*}
Let $C(\aA) \subset \Lambda$ be the image of 
$\ch|_{\aA}$. 
We also define $\epsilon_{\sigma}(v) \in H(\aA)$ in the following way: 
\begin{align}\label{sum:e}
\epsilon_{\sigma}(v) \cneq \sum_{l\ge 1}\, \sum_{\begin{subarray}{c}
v_i \in C(\aA), \ 1\le i\le l \\ 
v_1+ \cdots +v_l=v \\
\arg Z(v_i)=\arg Z(v)
\end{subarray}}
\frac{(-1)^{l-1}}{l} \
\delta_{\sigma}(v_1) \ast \cdots \ast \delta_{\sigma}(v_l). 
\end{align}
By the support property and the boundedness of $\sigma$, the 
sum (\ref{sum:e}) is a finite sum, and so 
it is well-defined.
Then the argument of~\cite[Theorem~8.7]{Joy3}
shows that $\epsilon_{\sigma}(v) \in H^{\rm{Lie}}(\aA)$. 
Following the construction of 
generalized DT invariants in~\cite{JS}, we 
give the following definition: 
\begin{defi}\label{def:DTA}
For $v \in \Lambda$, 
we define the invariant $\DT_{\sigma}(v) \in \mathbb{Q}$
in the following way:
if $v \in C(\aA)$, we define 
it by the formula
\begin{align*}
\Pi \epsilon_{\sigma}(v)=-
\DT_{\sigma}(v) \cdot c_v. 
\end{align*}
Otherwise, we set
\begin{align*}
\DT_{\sigma}(v) \cneq \left\{ \begin{array}{cc}
\DT_{\sigma}(-v), & \mbox{ if } -v \in C(\aA), \\
0, & \mbox{ if } \pm v \notin C(\aA). 
\end{array} \right. 
\end{align*}
\end{defi}
As a summary, we have the following result: 
\begin{thm}\label{thm:DTstab}
Let $X$ be a smooth projective Calabi-Yau 3-fold 
satisfying Conjecture~\ref{conj:BMT}. 
Then for each $v \in \Lambda$, 
there is a map
\begin{align*}
\DT_{\ast}(v) \colon 
\Stab_{\omega, B}^{\circ}(X) \to \mathbb{Q}
\end{align*}
such that $\DT_{\sigma}(v)$ virtually counts 
$\sigma$-semistable objects 
$E \in D^b \Coh(X)$ with $\ch(E)=v$.    
If any $\sigma$-semistable object $E$ with $\ch(E)=v$
is $\sigma$-stable, then $\DT_{\sigma}(v)$
coincides with (\ref{DTwB2}). 
\end{thm}
\begin{proof}
If $\sigma$ is good, then 
the invariant $\DT_{\sigma}(v)$ 
is defined in Definition~\ref{def:DTA}.
Suppose $\sigma$ is not good. 
Since the set of good points in $\Stab_{\omega, B}^{\circ}(X)$
is dense, and the walls are defined over rational numbers, one 
can perturb $\sigma$ to a good stability condition 
$\sigma'$ so that 
stable factors of objects in $\mM_{\sigma}(v)$ and 
those in $\mM_{\sigma'}(v)$ are the same. 
We define $\DT_{\sigma}(v)$ to be 
$\DT_{\sigma'}(v)$, which is obviously 
independent of $\sigma'$.  
\end{proof}


\providecommand{\bysame}{\leavevmode\hbox to3em{\hrulefill}\thinspace}
\providecommand{\MR}{\relax\ifhmode\unskip\space\fi MR }
\providecommand{\MRhref}[2]{%
  \href{http://www.ams.org/mathscinet-getitem?mr=#1}{#2}
}
\providecommand{\href}[2]{#2}

\end{document}